\documentclass[12pt]{article}
\usepackage{amscd,amssymb,amsmath,latexsym,enumerate,amsthm}
\usepackage[all]{xy}
\usepackage[mathscr]{euscript}
\usepackage{epsfig}
\usepackage{epsfig,psfrag}
\usepackage{float}
\usepackage{graphicx} 
\title{Tiling Groupoids And Bratteli Diagrams\footnote{Work supported by the NSF grants no. DMS-0300398 and no. DMS-0600956.}}
\author{J. Bellissard, A. Julien, J. Savinien}

\date{ }

\newtheorem{theo}{Theorem}[section]
\newtheorem{defini}[theo]{Definition}
\newtheorem{proposi}[theo]{Proposition}
\newtheorem{lemma}[theo]{Lemma}
\newtheorem{coro}[theo]{Corollary}
\newtheorem{rem}[theo]{Remark}

\newtheorem{hypo}[theo]{Hypothesis}

\newcommand{\Bb}{{\mathcal B}}

\newcommand{\Ee}{{\mathcal E}}
\newcommand{\Ff}{{\mathcal F}}

\newcommand{\Hh}{{\mathcal H}}

\newcommand{\Ll}{{\mathcal L}}

\newcommand{\Vv}{{\mathcal V}}

\newcommand{\NM}{{\mathbb N}}

\newcommand{\RM}{{\mathbb R}}

\newcommand{\ZM}{{\mathbb Z}}

\newcommand{\htt}{{\hat{t}}}

\newcommand{\ts}{{\mathscr T}}

\newcommand{\vs}{{\mathscr V}}

\newcommand{\Cs}{$C^{\ast}$-algebra }         
\newcommand{\Css}{$C^{\ast}$-algebras }       
\newcommand{\CsS}{$C^{\ast}$-algebras}             
\newcommand{\supp}{\mbox{\rm supp}}                

\newcommand{\eaf}{\stackrel{\mbox{\tiny\sc AF}}{\sim}}
\newcommand{\dist}{\text{\rm dist}}                
\newcommand{\col}{\text{\rm Col}}                  
\newcommand{\punc}{\text{\rm punc}}                
\newcommand{\eps}{\varepsilon}
\newcommand{\op}{\text{\rm op}}                    

\begin{document}

\maketitle

\begin{abstract}
\noindent Let $T$ be an aperiodic and repetitive tiling of $\RM^d$ with finite local complexity.
Let $\Omega$ be its tiling space with canonical transversal $\Xi$.
The tiling equivalence relation $R_\Xi$ is the set of pairs of tilings in $\Xi$ which are translates of each others, with a certain (\'etale) topology.
In this paper $R_\Xi$ is reconstructed as a generalized ``tail equivalence'' on a Bratteli diagram, with its standard $AF$-relation as a subequivalence relation.

\vspace{.1cm}

\noindent Using a generalization of the Anderson--Putnam complex \cite{BBG06} $\Omega$ is identified with the inverse limit of a sequence of finite $CW$-complexes. A Bratteli diagram $\Bb$ is built from this sequence, and its set of infinite paths $\partial \Bb$ is homeomorphic to $\Xi$.
The diagram $\Bb$ is endowed with a horizontal structure: additional edges that encode the adjacencies of patches in $T$. This allows to define an \'etale equivalence relation $R_\Bb$ on $\partial \Bb$ which is homeomorphic to $R_\Xi$, and contains the $AF$-relation of ``tail equivalence''.
\end{abstract}


\tableofcontents


 \section{Introduction}
 \label{bratteli09.sect-intro}

 This article describes a combinatorial way to reconstruct a tiling space and its groupoid, using Bratteli diagrams. In a forthcoming paper, this description will be used for tilings given by multi-substitutions \cite{BJS09}.

 \vspace{.3cm}

  \subsection{Results}
  \label{bratteli09.ssect-res}

 Given a repetitive, aperiodic tiling $T$ with finite local complexity (FLC) in $\RM^d$, its tiling space $\Omega$, or hull, is a compact space obtained by taking a suitable closure of the family of tilings obtained by translating $T$. By construction, the translation group $\RM^d$ acts on $\Omega$ by homeomorphisms making the pair $(\Omega,\RM^d)$ a topological dynamical system.
 It is well known that repetitivity (or \emph{uniform repetitivity}, as it is called in the symbolic one-dimensional case) is equivalent to the minimality of this action, see~\cite{Qu87,RW92}.
 Equivalently, this dynamical system can be described through a groupoid \cite{Co79,Re80,Co94}, denoted by $\Omega\rtimes \RM^d$, called the {\em crossed product} of the tiling space by the action.
If the tiles of $T$ are punctured, the subset of $\Omega$ made of tilings with one puncture at the origin of $\RM^d$ is a compact subset $\Xi$ called the {\em canonical transversal} \cite{Be86}.
For quasi-crystals this identifies with the so-called {\em atomic surface} \cite{BHZ00,LectQC}.
It has been shown that FLC implies that the transversal is completely disconnected \cite{Ke97}, while aperiodicity and minimality eliminate isolated points, making it a Cantor set. Similarly, there is an {\em \'etale} groupoid $\Gamma_\Xi$ associated with it \cite{Co79,Re80}, which plays a role similar to the Poincar\'e first return map in usual dynamical systems. This groupoid is defined by the equivalence relation $R_\Xi$ identifying two tilings of the transversal differing by a space translation.

 \vspace{.1cm}

 While the existence of the hull does not require much information, it can be quite involved to have an effective description which allows computations.
 The first step in this direction came from the work of Anderson and Putnam for substitution tilings \cite{AP98}.
 A substitution is a rule describing how each tiles, after suitable rescaling, is covered by other tiles touching along their faces.
 The most publicized example is the Penrose tiling in its various versions, like kites and darts \cite{Pen74}.
Anderson and Putnam built a $CW$-complex $X$ of dimension $d$, with $d$-cells given by suitably decorated prototiles and showed that the substitution induces a canonical map $\phi:X \rightarrow X$.
Then the hull can be recovered as the inverse limit $\Omega =\varprojlim (X,\phi)$.
This construction has several generalizations for repetitive, aperiodic, and FLC tilings without a substitution rule, see~\cite{BBG06,Gah,Sad03}.
It is proved that there is a sequence of finite $CW$-complexes $(X_n)_{n\in\NM}$ of dimension $d$ and maps $\phi_n:X_{n+1}\rightarrow X_n$ so that the hull is given by the inverse limit $\Omega =\varprojlim (X_n,\phi_n)$.
It is even possible to choose these $CW$-complexes to be smooth branched manifolds~\cite{BBG06}.

 \vspace{.1cm}

The first goal of the present paper is to describe the construction of a Bratteli diagram from these data in Section \ref{bratteli09.ssect-brat}. A Bratteli diagram is a graph with a marked vertex $\circ$ called the {\em root}. The set of vertices $\vs$ is graded by a natural integer, called the {\em generation}, so that $\vs= \bigcup_{n\in\NM}\vs_n$ with $\vs_0=\{\circ\}$ and $\vs_n\cap\vs_m=\emptyset$ if $n\neq m$. Then edges exist only between $\vs_n$ and $\vs_{n+1}$. In the present construction, each vertex in $\vs_n$ is given by a $d$-cell of $X_n$. Then there is an edge between a vertex $v\in\vs_n$ and $w\in\vs_{n+1}$ if and only if the $d$-cell $v$
can be found inside $w$. Each edge will be  labeled by the translation vector between the puncture of $w$ and the one of $v$ inside $w$. Namely at each generation $n$, the Bratteli diagram encodes how a $d$-cell of $X_{n+1}$ is filled by $d$-cells of $X_n$. More precisely it encodes not only which cells in $X_n$ occur but also their position relative to the $d$-cell of $X_{n+1}$. As in \cite{BBG06}, since the family $(X_n)_{n\in\NM}$ is not unique, there are several choices to build a Bratteli diagram associated with a tiling space. 

 \vspace{.1cm}

Conversely, starting from such a Bratteli diagram, a tiling can be described as an infinite path starting from the root. Such a path describes how a given patch at generation $n$, is embedded in a larger patch at generation $n+1$. By induction on the generations, an entire tiling is reconstructed. In addition, the origin of the tiling is defined uniquely at each step, thanks to the label of the edges involved. If each patch involved in the construction is decorated by its {\em collar}, namely provided that the tiling {\em forces its border}\footnote{a notion initially introduced by Kellendonk for substitution tilings~\cite{Ke97}} in the sense of \cite{BBG06}, the tiling obtained eventually covers $\RM^d$, even if the origin is at a fixed distance from the boundary of the patch, because the collar, beyond this boundary, increases in size as well. As a result the transversal $\Xi$ is recovered, as a compact topological space, from the space of infinite rooted paths in the diagram (Theorem \ref{bratteli09.thm-isotrans}).

 \vspace{.1cm}

In a Bratteli diagram, the {\em tail equivalence} identifies two infinite paths differing only on a finite number of edges (Definition \ref{bratteli09.def-AFER}). It gives rise to a groupoid called the {\em AF-groupoid} of the diagram. In the language of tiling spaces, this groupoid describes the translation structure inside the tiling, up to an important obstacle. Namely, a tiling built from an infinite path with the origin at a fixed distance from the boundary of the corresponding patches cannot be identified, modulo the tail equivalence, with a tiling built from a path with origin at a fixed distance from the same boundary but located on the other side of it. As a result, a tiling obtained in this way will be subdivided into regions, that will be called {\em AF-regions}  here (Definition \ref{bratteli09.def-AFregion}), separated by boundaries. Note that Matui has found similar features for a class of $2$-dimensional substitution tilings in \cite{Ma06}. 
By contrast, the groupoid of the transversal $\Gamma_\Xi$ (Definition \ref{bratteli09.def-gpd}) allows to identify two such regions through translation.
So that recovering $\Gamma_\Xi$ from the Bratteli diagram requires to change the definition of the tail equivalence. The present paper offers a solution to this problem by adding a {\em horizontal structure} to the diagram making it a {\em collared Bratteli diagram} (Definition \ref{bratteli09.def-cBrat}). Its aim is to describe in a combinatorial way, from local data, how to locate a tile in a patch, relative to its boundary. Practically it consists in adding edges between two vertices of the same generation, describing pairs of tiles, in a pair of patches, each in the collar of one another (see Figure \ref{bratteli09.fig-horizontal}). These edges are labeled by the translation between these tiles. In other words, the horizontal edges describe how to glue together two patches in a tiling across an AF-boundary.
Then it becomes possible to extend the tail equivalence into an \'etale equivalence relation on the collared Bratteli diagram (Definition \ref{bratteli09.def-eqrel} and Theorem \ref{bratteli09.def-etale}) in order to recover the groupoid of the transversal (Theorem \ref{bratteli09.thm-eqrel} and Corollary \ref{bratteli09.cor-groupoid}), with the AF-groupoid as a sub-groupoid (Remark \ref{bratteli09.rem-AFcomp}).

\vspace{.1cm}

\noindent In the particular case of $1$-dimensional tilings, this larger equivalence relation is generated by the AF-relation and finitely many pairs of minimal and maximal paths in the diagram, which are derived from the collared diagram in a natural and explicit way (Proposition \ref{bratteli09.prop-1d}).
The examples of the Fibonacci and Thue-Morse tilings are illustrated in details in Section~\ref{bratteli09.sect-examples}.
This study of $1$-dimensional tilings also allows to view the paths whose associated tilings have their punctures at a minimum distance to an AF-boundary (Corollary \ref{bratteli09.cor-Fsigma}) as generalizations of extremal paths for $1$-dimensional systems.

 \vspace{.3cm}

  \subsection{Background}
  \label{bratteli09.ssect-hist}

 \noindent This work gives one more way of describing tilings or their sets of punctures, and their groupoids, liable to help describing various properties of tiling spaces, such as their topology or their geometry. It benefited from almost thirty years of works with original motivation to describe more precisely the properties of aperiodic solids. In particular the notion of hull, or tiling space, was described very early as a fundamental concept \cite{Be86} encoding their macroscopic {\em translation invariance}, called homogeneity. See for instance \cite{Be93,Be01} for reviews and updates.

 \vspace{.1cm} 

 \noindent During the eighties, the discovery of quasicrystals \cite{SBGC84} was a landmark in this area and stimulated a lot of mathematical works to describes this new class of materials \cite{LectQC,Se95}. In particular, it was very convenient to represent the structure of these materials by various examples of tilings, such as the Penrose tiling or its $3$-dimensional analogs \cite{Pen74,Kr82,LectQC}. It was shown subsequently \cite{KG95} that such structures are liable to describe precisely the icosahedral phase of quasicrystalline alloys such as \mbox{$\mathrm{Al}\mathrm{Cu}\mathrm{Fe}$}.

 \vspace{.1cm}

 \noindent In the end of the eighties, Connes attracted the attention of mathematicians to the subject by showing that the Penrose tiling was a typical example of a noncommutative space~\cite{Co94}. It is remarkable that already then, Connes used a Bratteli diagram to encode the combinatorics of patches between generations. Its is important to realize, though, that the construction given by Connes was based on the {\em substitution} proposed originally by Penrose \cite{Pen74} but ignored entirely the additional structure given by space {\em translations}. In the context of the present paper, including the translations is the key reason leading to the extension of the $AF$-relation. This leads to a non-$AF$ \Cs instead, a difficulty at the source of so many works during the last twenty years. 

 \vspace{.1cm}

 \noindent Bratteli created his diagrams in \cite{Bra72} to classify $AF$-algebras, namely \Css obtained as the unions of finite dimensional \CsS. It took two decades before it was realized that, through the notion of {\em Vershik map} \cite{Ve81} such diagrams could encode any minimal homeomorphism of the Cantor set \cite{Sk90,HPS92}. The corresponding crossed product \Cs was shown to characterize the homeomorphism up to orbit equivalence. This classification was a natural extension of a similar problem for ergodic actions of $\ZM$, a problem solved by Krieger and Connes within the framework of von Neumann algebras (see \cite{CK77} for instance). This program was continued until recently and lead to the proof of a similar result for minimal actions of $\ZM^d$ on the Cantor set \cite{For97,DHS99,GPS04,Phi05,GPS06,GMPSa08,GMPSc08}. In particular, as a consequence of this construction, Giordano--Matui--Putnam--Skau proved that any minimal $\ZM^d$-action on the Cantor set is orbit equivalent to a $\ZM$-action \cite{GMPSb08}. Moroever, the groupoid of the transversal of every aperiodic repetitive FLC tiling space is orbit equivalent to a minimal $\ZM$-action on the Cantor set \cite{AGL09}. As it turns out the formalism described in the present work is similar to and, to a certain extend inspired by, the construction of {\em refined tessellations} made by Giordano {\it et al.} in \cite{GMPSb08}. 

 \vspace{.5cm}

 \noindent This paper is organized as follows: the Section~\ref{bratteli09.sect-tilgpd} contains the basic definitions about tilings, groupoids, and the constructions existing in the literature required in the present work. Section~\ref{bratteli09.sect-gpdbrat} is dedicated to the definition and properties of the Bratteli diagram associated with a tiling. The main reconstruction theorems are reproduced there.
 The last Section~\ref{bratteli09.sect-examples} illustrates the construction for $1$-dimensional substitution tilings, and treats the two examples of the Fibonacci and Thue--Morse tilings in details.

\vspace{.3cm}

\noindent {\bf Acknowledgements: } The three authors are indebted to Johannes Kellendonk for his invaluable input while starting this work. J.S. would like to thank John Hunton for discussions and for communicating information about his research using the notion of Bratteli diagrams. This work was supported by the NSF grants no. DMS-0300398 and no. DMS-0600956, by the School of Mathematics at Georgia Tech in the Spring 2009 and by the SFB 701 (Universit\"at Bielefeld, Germany) in the group led by Michael Baake.

\newpage
\section{Tilings and tiling groupoids}
\label{bratteli09.sect-tilgpd}

This section is a reminder. The basic notions about tilings and their groupoids are defined and described. The finite volume approximations of tiling spaces by branched manifolds are also summarized \cite{BBG06}.

\subsection{Tilings and tiling spaces}
\label{bratteli09.ssect-tilspaces}

All tilings in this work will be subsets of the $d$-dimensional Euclidean space $\RM^d$. Let $B(x,r)$ denote the open ball of radius $r$ centered at $x$.

\begin{defini}
\label{bratteli09.def-tiling}
\begin{enumerate}[(i)]

\item A {\em tile} is a compact subset of $\RM^d$ which is the closure of its interior.

\item A {\em punctured tile} $t_x$ is an ordered pair consisting of a tile $t$ and a point $x\in t$.

\item A {\em partial tiling} is a collection $\{ t_i\}_{i\in I}$ of tiles with pairwise disjoint interiors. The set\/ \(\bigcup_{i\in I} t_i\) is called its {\em support}.

\item A {\em patch} is a finite partial tiling. A patch is punctured by the puncture of one of the tiles that it contains.

\item A {\em tiling} is a countable partial tiling with support $\RM^d$.
A tiling is said to be punctured if its tiles are punctured.

\item The {\em inner radius} of a tile or a patch is the radius of the largest ball (centered at its puncture) that is contained in its support.
The {\em outer radius} is the radius of the smallest ball (centered at its puncture) that contains its support.
\end{enumerate}
\end{defini}

\begin{hypo}
\label{tLap09.hypo-puncCW}
From now all tiles and tilings are punctured. 
In addition, each tile is assumed to have a finite $CW$-complex structure.
The $CW$-complex structures of the tiles in a tiling are compatible: the intersection of two tiles is a subcomplex of both.
In particular, the support of a tiling gives a $CW$-complex decomposition of $\RM^d$.
\end{hypo}

\begin{rem}
\label{bratteli09.rem-}
{\em
\begin{enumerate}[(i)]
\item Definition \ref{bratteli09.def-tiling} allows tiles and patches to be {\em disconnected}.
\item Tiles or patch are considered as subsets of $\RM^d$. If $t$ is a tile and $p$ a patch, the notations {\em $t\in T$ and $p\subset T$} mean {\em $t$ is a tile and $p$ is a patch of the tiling $T$, at the positions they have as subsets of $\RM^d$}.
\end{enumerate}}
\end{rem}

\noindent The results in this paper are valid for the class of tilings that are {\em aperiodic, repetitive}, and have {\em finite local complexity}.

\begin{defini} 
\label{bratteli09.def-repFLC}
Let $T$ be a tiling of $\RM^d$.

\begin{enumerate}[(i)]

\item $T$ has {\em Finite Local Complexity} (FLC) if for any $\rho>0$ there are up to translation only finitely many patches of outer radius less than $\rho$.

\item If $T$ has FLC then it is {\em repetitive} if given any patch $p$, there exists $\rho_p>0$ such that for every $x\in\RM^d$, there exists some $u\in\RM^d$ such that \( p+ u \in T\cap B(x,\rho_p)\).

\item For $a \in\RM^d$ let \( T +a = \{ t+a \, : \, t \in T \}\) denote the translate of $T$ by $a$. Then $T$ is {\rm aperiodic} if \(T+a \ne T\), for all vector $a\ne0$.

\end{enumerate}

\end{defini}

From now on, we will only consider tilings which satisfy the conditions above.
Note that FLC implies that there are only finitely many tiles up to translation.

One of the authors defined in \cite{BHZ00} a topology that applies to a large class of tilings (even without FLC).
In the present setting, this topology can be adapted as follows.
Let $\Ff$ be a family of tilings. 
Given an open set $O$ in $\RM^d$ with compact closure and an $\epsilon>0$, a neighborhood of a tiling $T$ in $\Ff$ is given by
\[
U_{O,\epsilon}(T) = \left\{ T' \in \Ff \, : \,
    \exists x,y \in B(0,\epsilon), \ (T \cap O) + x = (T \cap O) + y \right\},
\]
where $T \cap O$ is the notation for the set of all cells of $T$ which intersect~$O$.


\begin{defini}
\label{bratteli09.def-hull}

\begin{enumerate}[(i)]

\item The {\em hull} or {\em tiling space} of $T$, denoted $\Omega$, is the closure of  \( T+\RM^d\) for the topology defined above.

\item The {\em canonical transversal}, denoted $\Xi$, is the subset of $\Omega$ consisting of tilings having one tile with puncture at the origin $0_{\RM^d}$.

\end{enumerate}

\end{defini}

By FLC condition, the infimum of the inner radii of the tiles of $T$ is $r>0$.
So the canonical transversal is actually an abstract transversal for the $\RM^d$ action in the sense that it intersects every orbit, and $(\Xi + u) \cap \Xi =\emptyset$ for all $u\in\RM^d$ small enough ($|u|<2r$).
The hull of a tiling is a dynamical system \( (\Omega, \RM^d)\) which, for the class of tilings considered here, has the following well-known properties (see for example \cite{BBG06} section 2.3).

\begin{theo}
\label{bratteli09.thm-hull}
Let $T$ be a tiling of $\RM^d$.

\begin{enumerate}[(i)]
\item $\Omega$ is compact.

\item $T$ is repetitive if and only if the dynamical system \( (\Omega, \RM^d) \) is {\rm minimal}.

\item If $T$ has FLC, then its canonical transversal $\Xi$ is totally disconnected \cite{Ke97,BHZ00}.

\item If $T$ is repetitive and aperiodic, then $\Omega$ is {\em strongly aperiodic}, {\it i.e.} contains no periodic points.

\item If $T$ is aperiodic, repetitive, and has FLC, then $\Xi$ is a Cantor set.

\end{enumerate}

\end{theo}

By the minimality property, if $\Omega$ is the hull of a tiling $T$, any patch $p$ of any tiling $T' \in \Omega$ appears in $T$ in some position.
In case (iv), the sets
\begin{equation}
\label{bratteli09.eq-clopen}
\Xi(p) = \bigl\{ T' \in \Xi \ : \ T' \text{ contains (a translate of) $p$ at the origin}
\bigr\}\,,
\end{equation}
for $p \subset T$ a patch of $T$, form a base of clopen sets for the topology of $\Xi$.

\begin{rem}
\label{bratteli09.rem-topohull}
{\em A metric topology for tiling spaces has been used in the literature.
Let $T$ be a repetitive tiling or $\RM^d$ with FLC.
The orbit space of $T$ under translation by vectors of~$\RM^d$, \(T+\RM^d\), is endowed with a metric as follows (see \cite{BBG06} section 2.3).
For $T_1$ and $T_2$ in \(T+\RM^d\), let $A$ denote the set of $\varepsilon$ in $(0,1)$ such that there exists \(a_1,a_2 \in B(0, \varepsilon ) \) for which \(T_1+a_1\) and \(T_2+a_2\) agree on \(  B(0, 1/\varepsilon ) \), {\it i.e.} their tiles whose punctures lie in the ball are matching, then
\[
 \delta (T_1,T_2) = \min \big(\inf A, 1 \big).
\]
Hence the diameter of  \(T+\RM^d\) is bounded by $1$.
With this distance, the action of $\RM^d$ is continuous.

For the class of repetitive tilings with FLC, the topology of the hull given in definition \ref{bratteli09.def-hull} is equivalent to this $\delta$-metric topology \cite{BBG06}.}
\end{rem}

\subsection{Tiling equivalence relations and groupoids}
\label{bratteli09.ssect-gpd}

Let $\Omega$ be the tiling space of an aperiodic, repetitive, and FLC tiling of $\RM^d$, and let $\Xi$ be its canonical transversal.

\begin{defini}
\label{bratteli09.def-ERhull}
The equivalence relation $R_\Omega$ of the tiling space is the set
\begin{equation}
\label{bratteli09.eq-ERhull}
R_\Omega = \bigl\{ \
(T,T') \in \Omega \times \Omega \ : \ \exists a \in \RM^d\,, \; T' = T + a \
\bigr\}
\end{equation}
with the following topology: a sequence $(T_n, T'_n = T_n + a_n)$ converges to $(T,T'=T+a)$ if \( T_n \rightarrow T\) in $\Omega$ and \(a_n \rightarrow  a\) in $\RM^d$.

The equivalence relation of the transversal is the restriction of $R_\Omega$ to $\Xi$: 
\begin{equation}
\label{bratteli09.eq-ERtrans}
R_\Xi = \bigl\{
(T,T') \in \Xi \times \Xi \ : \ \exists a \in \RM^d\,, \; T' = T + a \
\bigr\}
\end{equation}
\end{defini}

\noindent Note that the equivalence relations are {\em not} endowed with the relative topology of \(R_\Omega \subset \Omega \times \Omega\) and \(R_\Xi \subset \Xi \times \Xi\). For example, by repetitivity, for $a$ large, $T$ and $T+a$ might be close to each other in $\Omega$, so that $(T,T+a)$ is close to $(T,T)$ for the relative topology, but not for that from $\Omega \times \RM^d$. The map \( (T,a) \mapsto (T, T+a)\) from  $\Omega \times \RM^d$ to $\Omega \times \Omega$ has a dense image, coinciding with $R_\Omega$ and is one-to-one because $\Omega$ is strongly aperiodic (contains no periodic points). The topology of $R_\Omega$ is the topology induced by this map.

\begin{defini}
\label{bratteli09.def-etale}
An equivalence relation $R$ on a compact metrizable space $X$ is called {\em \'etale} when the following holds.
\begin{enumerate}[(i)]

\item The set \(R^2 = \{((x,y),(y,z)) \in R\times R \}\) is closed in $R\times R$ and the maps sending $((x,y),(y,z))$ in $R\times R$ to $(x,y)$ in $R$, $(y,z)$ in $R$, and $(x,z)$ in $R$ are continuous.

\item The diagonal \( \Delta(R)=\{ (x,x) : x\in X\}\) is open in $R$.

\item The range and source maps \( r,s : R\rightarrow X\) given by \(r(x,y)=x, s(x,y)=y\), are open and are local homeomorphisms.
\end{enumerate}
A set $0\subset R$ is called an {\em $R$-set}, if $O$ is open in $R$, and $r\vert_O$ and $s\vert_O$ are homeomorphisms.
\end{defini}

It is proved in \cite{Ke95} that $R_\Xi$ is an \'etale equivalence relation.

A {\em groupoid} \cite{Re80} is a small category (the collections of objects and morphisms are sets) with invertible morphisms.
A topological groupoid, is a groupoid $G$ whose sets of objects $G^0$ and morphisms $G$ are topological spaces, and such that the composition of morphisms \( G\times G \rightarrow G\), the inverse of morphisms \(G\rightarrow G\), and the source and range maps \(G\rightarrow G^0\) are all continuous maps.

Given an equivalence relation $R$ on a topological space $X$, there is a natural topological groupoid $G$ associated with $R$, with objects $G^0=X$, and morphisms $G=\{ (x,x') : x\sim_R x'\}$.
The topology of $G$ is then inherited by that of $R$.

\begin{defini}
\label{bratteli09.def-gpd}
The {\em groupoid of the tiling space} is the groupoid of $R_\Xi$, with set of objects $\Gamma_\Xi^0 = \Xi$ and morphisms
\begin{equation}
\label{bratteli09.eq-gpd}
\Gamma_\Xi = \bigl\{ (T,a) \in \Xi \times \RM^d \ : \ T+a \in \Xi \bigr\}\,.
\end{equation}
\end{defini}

There is also a notion of \'etale groupoids \cite{Re80}.
Essentially, this means that the range and source maps are local homeomorphisms.
It can be shown that $\Gamma_\Xi$ is an \'etale groupoid \cite{Ke95}.

\subsection{Approximation of tiling spaces}
\label{bratteli09.ssect-approxtil}

Let $T$ be an aperiodic, repetitive and FLC (punctured) tiling of $\RM^d$.
Since $T$ has finitely many tiles up to translation, there exists $R>r>0$, such that the minimum of the inner radii of its tiles equals $r$, and the maximum of the outer radii equals $R$. The set $T^\punc$ of punctures of tiles of $T$ is a {\em Delone set} (or $(r,R)$-Delone set).
That is $T^\punc$ is {\em uniformly discrete} (or $r$-uniformly discrete): for any $x\in \RM^d$, \(B(x,r) \cap T^\punc\) contains at most a point, and {\em relatively dense} (or $R$-relatively dense): for any $x\in \RM^d$, \(B(x,R) \cap T^\punc\) contains at least a point.

\begin{rem}
\label{bratteli09.rem-delone}
{\em Properties of aperiodicity, repetitivity, and FLC can be defined for Delone sets as well. See \cite{BHZ00,La99A,La99B,LP03} for instance.
The study of Delone sets is equivalent to that of tilings with the same properties. Any tiling $T$ defines the Delone set $T^\punc$.
Conversely, a Delone set $\Ll$ defines its {\em Voronoi tiling} as follows.
Let 
\[
v_x = \bigl\{ y \in \RM^d \ : \ |y-x| \le |z-x|\,, \forall z \in \Ll \bigr\}
\]
be the {\em Voronoi tile} at $x \in\Ll$ (it is a closed and convex polytope).
The Voronoi tiling $V(\Ll)$ is the tiling with tiles \( v_x, x \in \Ll\).
}
\end{rem}

\begin{defini}
\label{bratteli09.def-ctile}
Let $T$ be a tiling such that $T^\punc$ is $r$-uniformly discrete for some $r>0$.
\begin{enumerate}[(i)]

\item The {\em collar} of a tile $t \in T$ is the patch
\(
\col(t) = \bigl\{
t'\in T \ : \ \dist(t,t')\le r
\bigr\} \,.
\)

\item A {\em prototile} is an equivalence class of tiles under translation.
A {\em collared prototile} is the subclass of a prototile whose representatives have the same collar up to translation. Then $\ts(T)$ and $\ts_c(T)$ will denote the set of prototiles and collared prototiles of $T$ respectively.

\item The {\em support} $\supp [t]$ of a prototile $[t]$ is the set $t'-a'$ with $t'$ a representant in the class of $[t]$, and $a'\in\RM^d$ is the vector joining the origin to the puncture of $t'$. As can be checked easily, the set $t'-a'$ does not depend upon which tile $t'$ is chosen in $[t]$. Hence $\supp [t]$ is a set obtained from any tile in $[t]$ by a translation with puncture at the origin. The support of a collared prototile is the support of the tiles it contains.
\end{enumerate}
\end{defini}

A collared prototile is a prototile where a local configuration of its representatives has been specified: each representative has the same neighboring tiles up to translation. In general, $\col(t)$ may contain more tiles than just the ones intersecting $t$, as illustrated in Figure \ref{bratteli09.fig-collar}.

\begin{figure}[!ht]
  \begin{center}
  \includegraphics[width=9cm]{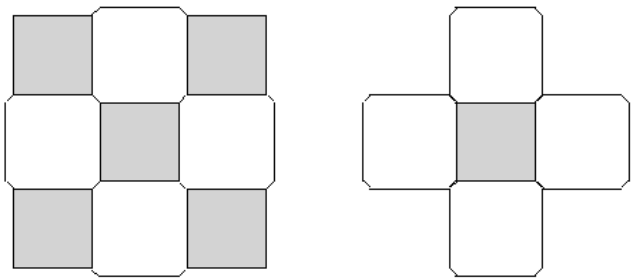}
  \end{center}
\caption{\small{$\col(t)$ {\it vs}  patch of tiles intersecting $t$.}}
\label{bratteli09.fig-collar}
\end{figure}

Since $T$ is FLC, it follows that its has only a finite number of collared prototiles. Moreover, any tiling in its hull have the same set of collared prototiles.


The previous description allows to represent a tiling $T$ in a more combinatorial way. Given a prototile $\htt$ (collared or not), let $T^\punc(\htt)$ be the subset of $T^\punc$ made of punctures corresponding to tiles in $\htt$ (it should be noted that it is also an aperiodic repetitive FLC Delone set). 

\begin{defini}
\label{bratteli09.def-combtil}
The set $I(T)$ of pairs $i=(a,\htt)$, where $\htt\in\ts(T)$ is a prototile and $a\in T^\punc(\htt)$, will be called the combinatorial representation of $T$. Similarly, the set $I_c(T)$ of pairs $i=(a,\htt)$, where $\htt\in\ts_c(T)$ is a collared prototile and $a\in T^\punc(\htt)$, will be called the combinatorial collared representation of $T$. 
\end{defini}

There is a one-to-one correspondence between the family of tiles of $T$ and $I_T$, namely, with each tile $t\in T$ is associated the pair $i(t)= (a(t),[t])$ where $a(t)$ is the puncture of $t$. By construction this map is one-to-one. Conversely $t_i$ will denote the tile corresponding to $i=(a,\htt)$. 

Thanks to the Hypothesis~\ref{tLap09.hypo-puncCW}, the tiles of $T$ are finite and compatible $CW$-complexes. The next definition is a reformulation of the Anderson--Putnam space \cite{AP98} that some of the authors gave in \cite{SB09}.
\begin{defini}
\label{bratteli09.def-AP}
Let \(\hat{t}_i, i = 1, \cdots p\) be the prototiles of $T$. Let $t_i$ be the support of $\hat{t}_i$. The {\em prototile space} of $T$ is the quotient $CW$-complex
\[
K(T) = \bigsqcup_{i=1}^n t_i / \sim \,,
\]
where two $k$-cells $c \in t_i^k$ and  $c' \in t_j^k$ are identified if there exists $u, u' \in \RM^d$ for which $t_i + u, t_j + u' \in T$, with  \(c + u = c' + u'\).

The {\em collared prototile space} $K^c(T)$ is built similarly out of the collared prototiles of $T$.
\end{defini}

\begin{proposi} 
\label{bratteli09.prop-projhull}
There is a continuous map \( \kappa^{(c)} : \Omega \rightarrow K^{(c)}(T) \).
\end{proposi}
\begin{proof} Let \( \lambda :  \sqcup \; t_j \rightarrow K^{(c)}(T) \) be the quotient map. And let \( \rho : \Omega \times \RM^d \rightarrow \sqcup \; t_j \) be defined as follows. If $a$ belongs to the intersection of $k$ tiles \(t^{\alpha_1}, \cdots t^{\alpha_k}\), in a tiling $T'\in\Omega$, with \(t^{\alpha_l} = t_{j_l} + u_{\alpha_l}(T'), \, l=1, \cdots k\), then
the point $a-u_{\alpha_l}(T')$ belongs to $t_{j_l}$. Moreover, all these points are identified after taking the quotient, namely $\lambda(a-u_{\alpha_l}(T'))=\lambda(a-u_{\alpha_m}(T'))$ if $1\leq l,m\leq k$. Therefore $\kappa(T',a) = \lambda(a-u_{\alpha_l}(T'))$ is well defined. This allows to set $\kappa^{(c)}(T') = \kappa(T',0_{\RM^d})$.
This map sends the origin of $\RM^d$, that lies in some tile of $T'$, to the corresponding tile $t_j$'s at the corresponding position.

In \(K^{(c)}(T)\), points on the boundaries of two tiles $t_i$ and $t_j$ are identified if there are neighboring copies of the tiles $t_i, t_j$ somewhere in $T$ such that the two associated points match.
This ensures that the map $\kappa^{(c)}$ is well defined, for if in $\RM^d$ tiled by $T'$, the origin belongs to the boundaries of some tiles, then the corresponding points in \(\sqcup \; t_j\) given by \(\rho (T, 0_{\RM^d})\) are identified by $\lambda$.

Let $a$ be a point in \(K^{(c)}(T)\), and $U_a$ and open neighborhood of $a$. Say $a$ belongs to the intersection of some tiles \(t_{j_1}, \cdots t_{j_k}\). Let $T'$ be a preimage of $a$: \(\kappa^{(c)} (T') = a\).
The preimage of $U_a$ is the set of tilings for which the origin lies in some neighborhood of tiles that are translates of \( t_{j_1}, \cdots t_{j_k}\), and this defines a neighborhood of $T'$ in the $\Omega$.
Therefore \(\kappa^{(c)}\) is continuous. 
\end{proof}

A {\em nested sequence of tilings} $(T_n)_n$, is a countable infinite sequence of tilings such that $T_1= T$ and for all $n \ge 2$ the tiles of $T_n$ are (supports of) patches of $T_{n-1}$. Without loss of generality, it will be assumed that $T \in \Xi$, {\it i.e.} has a puncture at the origin. To built such a sequence, the following procedure will be followed: assume that the tiling $T_{n-1}$ has been constructed and is aperiodic, repetitive and FLC with one puncture at the origin; then

\begin{enumerate}[(i)]

\item let \(p_n \subset T_{n-1}\) be a finite patch with a puncture at the origin, and let \(\Ll_n = \{ u \in \RM^d \ : \ p_n + u \subset T_{n-1}\}\)  be the Delone set of the punctures of the translated copies of $p_n$ within $T_{n-1}$;

\item let $V(\Ll_n)$ be the Voronoi tiling of $\Ll_n$ (Remark \ref{bratteli09.rem-delone}); to each tile of $T_{n-1}$ a tile $v\in V(\Ll_n)$ will be assigned (see the precise definition below);

\item for each tile $v\in V(\Ll_n)$, let $t^{(n)}_v$ be the union of the tiles of $T_{n-1}$  that have been assigned to it; then $T_n$ is defined as the tiling \(\{ t^{(n)}_v \}_{v \in V(\Ll_n)}\).

\end{enumerate}

It is worth remarking that $\Ll_n$ is an aperiodic, repetitive, FLC, Delone set as well. Thus the Voronoi tiling inherits these properties. 
Note however that in point (iii), two tiles $t_v^(n)$ may have the same shape. However,
if they correspond to different patches of $T_{n-1}$, they should be labeled as
different.

The second step of the construction above needs clarification since the tiles of  $V(\Ll_n)$ are not patches of $T_{n-1}$, but convex polytopes built out of $\Ll_n$ (see Remark \ref{bratteli09.rem-delone}).

First the set of patches of $T$ is countable, and second
for any patch $p\subset T$ the Voronoi tiling of \(\Ll_p = \{ u \in \RM^d \ : \ p + u \subset T_{n-1}\}\) has FLC, thus has finitely many prototiles.
Hence there is a vector $u\in\RM^d$ which is not in the span of any of the subspaces generated by the faces of the tiles $\Ll_p$ for all $p$.

A point $x\in \RM^d$ is called $u$-interior to a closed set $X\subset \RM^d$, and write \( x \stackrel{u}{\in} X\), if \( \exists \delta>0, \forall \eps \in (0,\delta), x+\eps u \in X \). Given a Voronoi tile $v\in V(\Ll_n)$ the patch associated with it is defined by (here $a(t)$ denotes the puncture of $t$)
\[
t_v^{(n)} = \cup \{ t \in T_{n-1} \ ; \ a(t) \stackrel{u}{\in} v \} \,.
\]
This gives an unambiguous assignment of tiles of $T_{n-1}$ to tiles of $T_n$.

There is therefore a natural subdivision of each tile of $T_n$ into tiles of $T_{n-1}$. Let $I_n=I(T_n)$ be the combinatorial representation of $T_n$ (see Definition~\ref{bratteli09.def-combtil}). There is a map $l_n : I_{n-1} \rightarrow I_n$, describing how to assign a tile in $T_{n-1}$ to the tile of $T_n$, namely to a patch of $T_{n-1}$ it belongs to. The ``inverse map'' defines a {\em substitution}, denoted by $\sigma_n$, namely a map from the set of tiles of $T_n$ to the set of patches of $T_{n-1}$ defined by 

$$\sigma_n(t_i^{(n)})= \{t_j^{(n-1)}\,;\, l_n(j)=i\}
$$

\noindent Such a substitution defines a map, also denoted by $\sigma_n$, from the $CW$-complex $T_n$ onto the $CW$-complex $T_{n-1}$, if the tiles $t_i^{(n)}$ are given the $CW$-complex structure inherited by the ones of the unions of the $t_j^{(n-1)}$ for $l_n(j)=i$.
In much the same way, this gives a map on the prototile space, and thus on the Anderson--Putnam complex as well.
A similar construction holds if tiles and prototiles are replaced by collared tiles and collared prototiles.
It leads to a canonical map $\sigma_n: K^c(T_n) \rightarrow K^c(T_{n-1})$, also denoted by $\sigma_n$. 

Let us denote by $r_{n}>0$ be the minimum of the inner radii of the tiles of $T_{n}$, and $R_n>0$ the maximum of the outer radii.
\begin{defini}
\label{bratteli09.def-nested}
A nested sequence of tilings $(T_n)_{n\in\NM}$ with substitution maps \(\sigma_n : T_n \mapsto T_{n-1}\) is called a {\em proper nested sequence}, if for all $n$ the tiling $T_n$ is aperiodic, repetitive, and has FLC, and the following holds for all $n\ge 2$:
\begin{enumerate}[(i)]

\item for each tile $t_n \in T_n$, $\sigma_n(t_n)$ is a patch of $T_{n-1}$,

\item for each tile $t_n \in T_n$, $\sigma_n(t_n)$ contains a tile of $T_{n-1}$ in the interior of its support,

\item there exists $\rho>0$ (independent of $n$), such that \(r_n>r_{n-1} +\rho\) and \(R_n>R_{n-1}+\rho\).

\end{enumerate}
\end{defini}
The nested sequence constructed above can be made proper by choosing, for each $n$, the patch $p_n$ to be large enough.

These conditions are sufficient to ensure that $K^c(T_n)$ is {\em zoomed out} of $K^c(T_{n-1})$ for all $n\ge 2$, in the terminology of \cite{BBG06} (Definition 2.43). The map  \(\sigma_n : K^c(T_n) \rightarrow K^c(T_{n-1})\) also {\em forces the border} (see \cite{Ke97}, Definition 15) because it is defined on the {\em collared} prototile spaces. This suffices to recover the tiling space from the sequence $\bigl( K^c(T_n) \bigr)_{n\in\NM}$.

\begin{theo}
\label{bratteli09.thm-isohull}
Let  $(T_n)_n$ be a proper nested sequence of tilings.
Then the tiling space $\Omega$ of $T_1$ is homeomorphic to the inverse limit of the complexes $K^c(T_n)$:
\[
\Omega \cong \varprojlim_{n\in\NM} \bigl( K^c(T_n), \sigma_n \bigr) \,.
\]
\end{theo}
This Theorem was first proved in \cite{BBG06} and it was also proved that the $\RM^d$-action could also be recovered from this construction (see also \cite{SB09} (Theorem 5)). The reader can also easily adapt the proof of Theorem \ref{bratteli09.thm-isotrans} given in section \ref{bratteli09.ssect-brat} to deduce Theorem \ref{bratteli09.thm-isohull}.

For each $n$, we can see the tiling $T_n$ as a ``subtiling'' of $T_1$.
The map  \(\kappa^{(c)} : \Omega \rightarrow K^{(c)}(T_1)\) of Proposition \ref{bratteli09.prop-projhull} can be immediately adapted to a map onto \(K^{(c)}(T_n)\) for each $n$.
\begin{proposi} 
\label{bratteli09.prop-projhulln}
There is a continuous map \(\kappa_n^{(c)} : \Omega \rightarrow K^{(c)}(T_n)\).
\end{proposi}
\begin{proof}
Let $T\in\Omega$. The origin of $\RM^d$ lies in some patch $p_i^{n} \subset T$ which is a translate of $\sigma_1 \circ \sigma_2 \circ \cdots \circ \sigma_n (t_i^{(n)})$ for some $i\in I_n$. We set \(\kappa_n^{(c)}\) to be the corresponding point in image of $t_i^{(n)}$ in $K^{(c)}(T_n)$.
\end{proof}

\section{Tiling groupoids and Bratteli diagrams}
\label{bratteli09.sect-gpdbrat}

\subsection{Bratteli diagrams associated with a tiling space}
\label{bratteli09.ssect-brat}

Bratteli diagrams were introduced in the seventies for the classification of $AF$-algebras \cite{Bra72}.
They were then adapted, in the topological setting, to encode $\ZM$-%
actions on the Cantor set \cite{For97}.
A specific case, close from our present concern, is the action of $\ZM$ by the shift on some closed, stable, and minimal subset of $\{0,1\}^\ZM$ (this is the one-dimensional symbolic analog of a tiling).
Then, Bratteli diagrams were used to represent $\ZM^2$ \cite{GMPSa08}, and recently $\ZM^d$ \cite{GMPSb08} Cantor dynamical systems, or to represent the transversals of substitution tiling spaces \cite{DHS99}.

\begin{defini}
\label{bratteli09.def-Bdiag}
A {\em Bratteli diagram} is an infinite directed graph \(\Bb=(\Vv,\Ee,r,s)\) with sets of vertices $\Vv$ and edges $\Ee$ given by
\[
\Vv = \bigsqcup_{n \in \NM \cup \{ 0 \} } \Vv_n \,, \quad \Ee = \bigsqcup_{n \in \NM} \Ee_n \,,
\]
where the $\Vv_n$ and $\Ee_n$ are finite sets.
The set $\Vv_0$ consists of a single vertex, called the {\em root} of the diagram, and noted $\circ$.
The integer $n \in \NM$ is called the {\em generation index}.
And there are maps 
\[
s : \Ee_n \rightarrow \Vv_{n-1} \,, \quad r : \Ee_n \rightarrow \Vv_{n}\,,
\]
called the {\em source} and {\em range} maps respectively.
We assume that for all $n \in \NM$ and $v \in \Vv_n$ one has \(s^{-1} (v) \ne \emptyset\) and \(r^{-1}(v) \ne \emptyset\) ({\em regularity}).

We call $\Bb$ {\em stationary}, if for all $n\in \NM$, the sets $\Vv_n$ are pairwise isomorphic and the sets $\Ee_n$ are pairwise isomorphic.
\end{defini}

Regularity means that there are there are no ``sinks'' in the diagram, and no ``sources'' apart from the root.

A Bratteli diagram can be endowed with a {\em label}, that is a map $l: \Ee \rightarrow S$ to some set $S$.

\begin{defini}
\label{bratteli09.def-paths}
Let $\Bb$ be a Bratteli diagram.
\begin{enumerate}[(i)]

\item A {\em path} in $\Bb$ is a sequence of composable edges \((e_n)_{1\le n \le m}, m\in \NM\cup\{+\infty\}\), with $e_n \in \Ee_n$ and $r( e_n) = s( e_{n+1})$.

\item If $m$ is finite, $\gamma=(e_n)_{n\le m}$, is called a {\em finite path}, and $m$ the length of $\gamma$.
We denote by $\Pi_m$ the set of paths of length $m$.

\item If $m$ is infinite, $x=(e_n)_{n\in\NM}$, is called an {\em infinite path}.
We denote by $\partial \Bb$ the set of infinite paths.

\item We extend the range map to finite paths: if \(\gamma = (e_n)_{n \le m}\) we set $s(\gamma) := s(e_m)$.
\end{enumerate}
\end{defini}

In addition to Definition \ref{bratteli09.def-Bdiag}, we ask that a Bratteli diagram satisfies the following condition.
\begin{hypo}
\label{bratteli09.hypo-split}
For all $v \in \Vv$, there are at least two distinct infinite paths through $v$.
\end{hypo}

The set $\partial \Bb$ is called the {\em boundary} of $\Bb$.
It has a natural topology inherited from the product topology on $\prod_{i=0}^{+\infty}{\Ee_i}$, which makes it a compact and totally disconnected set.
A base of neighborhoods is given by the following sets:
\[
[\gamma] = \{ x \in \partial \Bb \ ; \ \gamma \text{ is a prefix of } x\}.
\]

Hypothesis~\ref{bratteli09.hypo-split} is the required condition to make sure that there are no isolated points.
This implies the following.
\begin{proposi}
With this topology, $\partial \Bb$ is a Cantor set.
\end{proposi}

We now build a Bratteli diagram associated with a proper nested sequence of tilings.
\begin{defini}
\label{bratteli09.def-tilbrat}
Let $(T_n)_{n\in\NM}$ be a proper nested sequence of tilings.
Let \(\hat{t}^{(n)}_{i}, i=1, \cdots p_m\), be the collared prototiles of $T_n$, and $t_{i}^{(n)}$ the representative of $\hat{t}_{i}^{(n)}$ that has its puncture at the origin.

The Bratteli diagram \(\Bb = (\Vv, \Ee,r,s,u)\) associated with $(T_n)_{n\in\NM}$ is given by the following:
\begin{enumerate}[(i)]

\item \(\Vv_0 = \{ \circ\}\), and \(\Vv_n = \{ t_{i}^{(n)}, i=1, \cdots p_n\}, n\in\NM\),

\item \(\Ee_1 \cong \Vv_1\): \( e\in \Ee_1\) if and only if \( s(e) = \circ\) and \( r(e) = t_{i}^{(1)}\),

\item  $e \in \Ee_n$ with \( s(e) = t_{i}^{(n-1)}\) and \( r(e) = t_{j}^{(n)}\), if and only if there exists $a \in \RM^d$ such that \( t_i^{(n-1)} + a \) is a tile of the patch \( \sigma_n (t_j^{(n)})\),

\item a label \( u : \Ee \rightarrow \RM^d\), with \(u(e)=  -a\) for \(e\in \Ee_{n\ge 2}\) (and $u=0$ on $\Ee_1$).
\end{enumerate}
\end{defini}

Figure \ref{bratteli09.fig-edge} illustrates condition (iii).
\begin{figure}[!ht]
  \begin{center}
  \includegraphics[width=8cm]{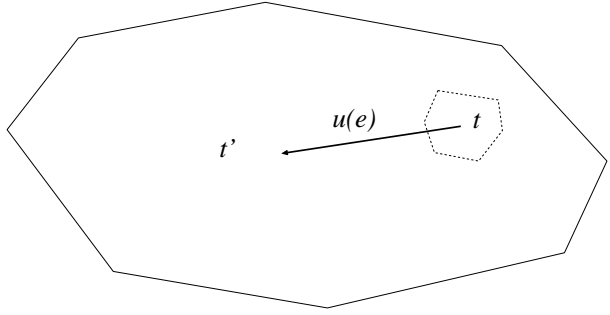}
  \end{center}
\caption{\small{Illustration of an edge $e\in\Ee$, with \(s(e)=t\) and \(r(e)=t'\).}}
\label{bratteli09.fig-edge}
\end{figure}

We extend the label as a map on finite paths \( u : \Pi \rightarrow \RM^d\):
for \(\gamma = (e_1, \cdots e_n)\in \Pi_n\) we set
\[
u(\gamma) = \sum_{i=1}^n u(e_i) \,.
\]
We can associate to each finite path $\gamma = (e_1, \cdots e_n)$ in $\Bb$, with 
\(s(\gamma)=t\), the patches of $T_1$
\begin{equation}
\label{bratteli09.eq-pathpatch}
p_\gamma = \sigma_1 \circ \sigma_2 \circ \cdots \sigma_n 
\bigl( t \bigr) + u(\gamma)\,, \quad \text{and} \quad
p^c_\gamma = \sigma_1 \circ \sigma_2 \circ \cdots \sigma_n 
\bigl( \col(t) \bigr) + u(\gamma)\,,
\end{equation}
where $p^c_\gamma$ is a ``collared patch'',  in the sense that it is the set of tiles of $T_1$ which make up the collar of a tile of $T_n$.

\begin{theo}
\label{bratteli09.thm-isotrans}
Let $\Xi$ be the transversal of the tiling space of $T_1$.
There is a canonical homeomorphism \( \varphi : \partial \Bb \rightarrow \Xi\).
\end{theo}
\begin{proof}
Let \(x \in \partial \Bb\), and set \( \gamma_n = x \vert_{\Pi_n}\).
Define
\[
\varphi(x) = \bigcap_{n\in\NM} \Xi (p^c_{\gamma_n}) \,,
\]
where \( \Xi( p^c_{\gamma_n})\) is the clopen set of tilings that have the patch $p^c_{\gamma_n}$ at the origin (see equations \eqref{bratteli09.eq-clopen} and \eqref{bratteli09.eq-pathpatch}).
Since $\Xi$ is compact, and \( \Xi ( p^c_{\gamma_n} )\) is closed and contains \(\Xi ( p^c_{\gamma_{n+1}})\) for all $n$, by the finite intersection property $\varphi(x)$ is a closed and non-empty subset of $\Xi$.
Let us show it consists of a single tiling.
Let then \(T,T' \in \varphi(x)\).
For all $n$, \(p^c_{\gamma_n} \subset T,T'\).
And since $p^c_{\gamma_n}$ is a collared patch, it contains a ball of radius $r_n$ (see Definition \ref{bratteli09.def-nested}) centered at its puncture.
Using the metric of Remark \ref{bratteli09.rem-topohull} this implies that \(\delta(T,T') \le 1/r_n\) for all $n$.
By condition (iii) in Definition \ref{bratteli09.def-nested}, \(r_n\rightarrow \infty\) as \( n \rightarrow \infty\), therefore \(\delta(T,T')=0\) and $T=T'$.

If $x\ne x'$ then $\gamma_n \ne \gamma'_n$ for some $n$, thus \(\Xi ( p^c_{\gamma_n} ) \cup \Xi ( p^c_{\gamma'_n}) = \emptyset\) and thus \( \varphi(x) \ne \varphi(x')\).
This proves that $\varphi$ is injective.

To prove that it is onto we exhibit an inverse.
Let $T\in \Xi$.
For each $n$, \(\kappa_n^c(T)\) (Proposition \ref{bratteli09.prop-projhulln}) lies in the interior of some tile $t_i^{(n)}$ and therefore we can associate with $T$ a sequence of edges through those vertices.
This defines an inverse for $\varphi$.

To prove that $\varphi$ and $\varphi^{-1}$ are continuous it suffices to show that the preimages of base open sets are open.
Clearly $\varphi([\gamma]) = \Xi(p^c_\gamma)$, hence $\varphi^{-1}$ is continuous.
Conversely, since $\Xi$ is compact and $\partial \Bb$ is Hausdorff, the
continuity of $\varphi$ is automatic.
\end{proof}

We now endow $\Bb$ with a {\em horizontal structure} to take into account the adjacency of prototiles in the tilings $T_n$, $n \in \NM$.
\begin{defini}
\label{bratteli09.def-cBrat}
A {\em collared Bratteli diagram} associated with a proper nested sequence \((T_n)_{n\in\NM}\), is a graph \(\Bb^c = (\Bb, \Hh)\), with $\Bb=(\Vv,\Ee,r,s,u)$ the Bratteli diagram associated with \((T_n)_{n\in\NM}\) as in Definition \ref{bratteli09.def-tilbrat}, and where $\Hh$ is the set of {\em horizontal edges}:
\[
\Hh = \bigsqcup_{n \in \NM} \Hh_n\,, \qquad \text{with}  \qquad r,s : \Hh_n \rightarrow \Vv_n \,,
\]
given by \(h \in \Hh_n\) with \(s(h) = t, r(h) = t'\), if and only if there exists \(a,a'\in\RM^d\), such that \(t+a, t'+a' \in T_n\), with
\[
t + a \in \col(t'+a') \,, \text{ and } t'+a' \in \col(t+a) \,, 
\]
and we extend the label $u$ to $\Hh$, and set \(u(h)=a'-a\).
\end{defini}
There is a horizontal arrow in $\Hh_n$ between two tiles $t,t' \in \Vv_n$, if one can find ``neighbor copies'' in $T_n$ where each copy belongs to the collar of the other.
In other words there exists a patch $p(t,t')$ ({\it i.e.} its tiles have pairwise disjoint interiors) with \(p(t,t')+a \subset T_n\) such that 
\[ 
\col(t) \cup \bigl( \col(t') - u(h) \bigr) \subset p(t,t')\,,
\]
see Figure \ref{bratteli09.fig-horizontal} for an illustration.
\begin{figure}[!ht]
  \begin{center}
  \includegraphics[width=8cm]{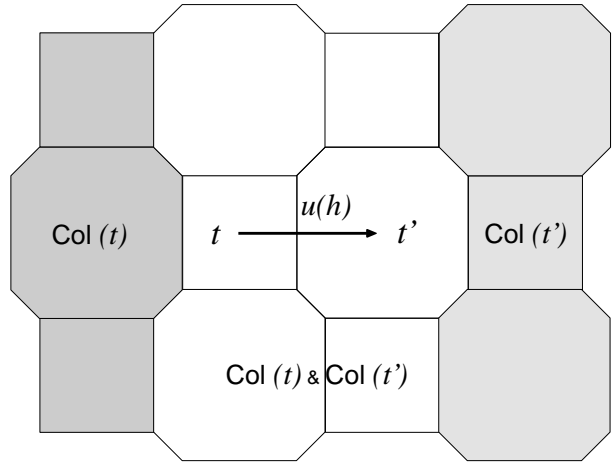}
  \end{center}
\caption{\small{Illustration of a horizontal edge \(h\in\Hh\) with \(s(h)=t\) and \(r(h)=t'\).}}
\label{bratteli09.fig-horizontal}
\end{figure}

For $h\in\Hh$, we define its {\em opposite edge} $h^\op$ by 
\[
s( h^\op) = r( h), \,  r(h^\op) = s(h) \,, \ \text{and} \ u(h^\op) = - u(h) \,.
\]
Clearly, for all $h$ in $\Hh$, $h^\op$ also belongs to $\Hh$, and $(h^\op)^\op = h$.
Also, the definition allows {\em trivial edges}, that is edges $h$ for which $s(h) = r(h)$ and $u(h)=0$.

\subsection{Equivalence relations}
\label{bratteli09.ssect-eqrel}

\begin{defini}
\label{bratteli09.def-AFER}
Let $\Bb$ be a Bratteli diagram and let
\[
R_n = \bigl\{
(x, \gamma) \in \partial \Bb \times \Pi_n \, : \, r( x\vert_{\Pi_n}) = r( \gamma) \bigr\}\,.
\]
with the product topology (discrete topology on $\Pi_n$).

The {\em $AF$-equivalence relation} is the direct limit of the $E_n$ given by
\[
R_{AF} = \varinjlim_n R_n =
\bigl\{
\bigr((e_n)_{n\in\NM}, (e'_n)_{n\in\NM} \bigr) \in \partial \Bb \times \partial \Bb \, : \, \exists n_0 \ \forall n\ge n_0 \ e_n = e'_n
\bigr\}\,,
\]
with the direct limit topology.
For $(x, y) \in R_{AF}$ we write $x \eaf y$ and say that the paths are {\em tail equivalent}.
\end{defini}

It is well known that $R_{AF}$ is an $AF$-equivalence relation, as the direct limit of the compact \'etale relations $R_n$, see \cite{Phi05}.

Assume now that $\Bb$ is a Bratteli diagram associated with a proper sequence of nested tilings $(T_n)_{n\in\NM}$.
We show now what this $AF$-equivalence relation represents for $T_1$ and its transversal $\Xi$.

For a finite path $\gamma \in \Pi_n$, let us write
\begin{equation}
\label{bratteli09.eq-pathtile}
t_{\gamma} = r(\gamma) + u(\gamma)\,,
\end{equation}
which is the support of the patch \(p_\gamma\) as defined in equation \eqref{bratteli09.eq-pathpatch}: \( \sigma_1 \circ \sigma_2 \circ \cdots \sigma_n (t_\gamma) = p_{\gamma}\).
For $x \in\partial \Bb$, we can see $t_1=t_{x\vert_{\Pi_1}}$ as a subset of $t_n=t_{x\vert_{\Pi_n}}$ for all $n\ge 2$.
We characterize the subset of $\partial \Bb$ for which $t_1$ stays close to the boundary of $t_n$ for all $n$, and its complement.
Recall that a $G_\delta$ is a countable intersection of open sets, and an $F_\sigma$ a countable union of closed sets.

\begin{lemma}\label{bratteli09.lem-minimal}
For any $n \in \NM$, there exists a $k > 0$ such that for any $v \in \Vv_n$ and
any $v' \in \Vv_{n+k}$, there is a path in $\partial \Bb$ from $v$ to $v'$.

In particular, any $x \in \partial \Bb$, the $AF$-orbit of $x$ is dense.
\end{lemma}

\begin{proof}
The definition of the $AF$ topology and the repetitivity of the underlying
tilings are the two key elements of this proof.
First, we prove that for any $v \in \Vv_n$, there exists $k \in \NM$ such that
for any $v' \in \Vv_{n+k}$, there is a path from $v$ to $v'$.
This is repetitivity: a vertex $v$ corresponds to a tile in $T_n$.
By repetitivity, any tile of $T_n$ appears within a prescribed range, say $R$.
Now, pick $k$ such that the inner radius of the tiles of $T_{n+k}$ is greater
than $R$.
It means that any tile of $T_n$ appears in (the substitute) of any tile of
$T_{n+k}$.
This is exactly equivalent to the existence of a path from any $v \in \Vv_n$ to
any $v' \in \Vv_{n+k}$.

Now, consider $x,y \in \Bb$.
Let us show that $y$ can be approximated by elements $x_n$ in $\Bb$ which are
all $AF$-equivalent to $x$.
Let $x_n$ be defined as follows: $(x_n)_{\vert \Pi_n} = y_{\vert \Pi_n}$.
We just proved that there is a $k$ such that there is a path from
$r(y_{\vert \Pi_n})$ to $r(x_{\vert \Pi_{n+k}})$.
Continue $x_n$ with this path, and define its tail to be the tail of $x$.
Then the sequence $(x_n)_{n \in \NM}$ is the approximation we were looking for.
\end{proof}

\begin{proposi} 
\label{bratteli09.prop-Gdelta}
The subset 
\[
G = \bigl\{ x \in \partial\Bb \ : \ \lim_{n\rightarrow + \infty} \dist( t_1, \partial t_n) = + \infty \bigr\}
\]
is a {\em dense $G_\delta$ in $\partial \Bb$}.
\end{proposi}

\begin{proof}
 For $m \in \NM$, let \(G_m = \{x \in \partial \Bb \, : \, \exists n_0 \in \NM,
\, \forall n \geq n_0, \, \dist(t_1,\partial t_n) > m \}\).
Then $G = \bigcap_{m \in \NM}{G_m}$.
Show that every $G_m$ is a dense open set in $\partial \Bb$.
Remark that if for some $n_0$, $\dist(t_1,\partial t_{n_0}) > m$, then this
property holds for all $n > n_0$.

Let us first prove that $G_m$ is dense.
Let $n \in \NM$. Then there is a $k$ such that there is a path from any $v \in
\Vv_n$ to any $v' \in \Vv_{n+k}$, by Lemma~\ref{bratteli09.lem-minimal}.
Let $l$ be such that $R_{n+k+l} - R_{n+k} > m$, where $R_n$ is the outer radius
of the tiles of $T_n$.
Then, let $\gamma$ be a path from $\Vv_{n+k}$ to $\Vv_{n+k+l}$ corresponding to
the inclusion of a tile of $T_{n+k}$ in the middle of a path of $T_{n+k+l}$.
Now, for any path $\eta$ of length $n$, it is possible
to join $\eta$ to $\gamma$.
Extend then this path containing $\eta$ and $\gamma$ arbitrarily to an infinite
path $x$ in $\partial \Bb$.
Then $x$ satisfies $\dist(t_1,t_{n+k+l}) > \dist(t_{n+k},t_{n+k+l}) > R_{n+k+l}
- R_{n+k} > m$, so $x \in G_m$.
It proves that $G_m$ is non-empty.
Since we could do this construction for all $n \in \NM$ and all $\eta$ of
length $n$, it proves that $G_m$ is dense.

Finally, $G_m$ is open because if $x \in G_m$ and satisfies $\dist(t_1,t_{n_0}) >
m$, then the tail of $x$ after generation $n$ can be changed without changing this
property. It proves that $G_m$ contains a neighborhood around all of its
points, and so it is open.

It proves that $G$ is a $G_\delta$ as intersections of dense open sets. Since
$\partial \Bb$ is compact, it satisfies the Baire property and so $G$ is dense.
\end{proof}

\begin{coro} 
\label{bratteli09.cor-Fsigma}
The subset 
\[
F = \bigl\{ x \in \partial\Bb \ : \ \lim_{n\rightarrow + \infty} \dist( t_1, \partial t_n) < + \infty \bigr\}
\]
is a {\em dense $F_\sigma$ in $\partial \Bb$}.
\end{coro}
\begin{proof}
With the notation of Proposition \ref{bratteli09.prop-Gdelta} consider the closed set $F_m = G_m^c$.
We have $F_m \subset F_{m+1}$ and $F = \cup_{m\in \NM} F_m$, thus $F$ is an $F_\sigma$.

The proof of the density of $F$ in $\partial \Bb$ is similar to that for $G$ in Proposition \ref{bratteli09.prop-Gdelta}.
Let $x \in \partial \Bb$.
Fix $l\in \NM$ and consider the patch \(p_l = p_{x \vert_{\Pi_l}}\) as in equation \eqref{bratteli09.eq-pathpatch}.
By repetitivity of $T_1$ there exists $R_l$ such that $T_1$ has a copy of $p_l$ in each ball of radius $R_l$.
Hence there exists $n_l\in \NM$ such that for all $n\ge n_l$ the tiles of $T_n$ (viewed as patches of $T_1$ under the map \(\sigma_1 \circ \sigma_2 \circ \cdots \sigma_n\)) contain a copy of $p_l$ that lies within a distance $R_l$ to their boundaries :
\(\dist(t_l, \partial t_n) \le \dist(t_1, \partial t_n) < R_l\).
We can thus extend the finite path $x_{\vert \Pi_{n_l}}$ to $x' \in \partial\Bb$ such
that \(\dist(t'_1, \partial t'_n) < R_l\) for all $n\ge n_l$.
Set $x_l = x'$.
We clearly have $x_l \in F$.
The sequence $(x_l)_{l\in\NM}$ built in this way converges to $x$ in $\partial \Bb$.
This proves that $F$ is dense in $\partial \Bb$.
\end{proof}

{\bf Notation.}
We will use now the following notation:
\begin{equation}
\label{bratteli09.eq-notation}
T_x := \varphi(x)\,, \ \text{ for } x\in\partial\Bb\,, \qquad \text{and} \qquad
x_T:=\varphi^{-1}(T)\,, \ \text{ for } T\in\Xi\,,
\end{equation}
where $\varphi$ is the homeomorphism of Theorem \ref{bratteli09.thm-isotrans}.

For each $T \in\Xi$ the equivalence relation $\eaf$ induces an equivalence relation on $T^\punc$:
\[
a \; \dot{\sim} \; b \text{ in } T^\punc \ \iff \  x_{T - a} \;  \eaf \; x_{T - b} \,.
\]
\begin{defini}
\label{bratteli09.def-AFregion}
An $AF$-{\em region} in a tiling $T \in \Xi$ is the union of tiles whose punctures are $\dot{\sim}$-equivalent in $T^\punc$.
\end{defini}
\begin{proposi}
\label{bratteli09.prop-AFregion}
A tiling $T \in\Xi$ has a single $AF$-region if and only if \(x_T \in G\).
\end{proposi}
\begin{proof}
Assume $T\in\Xi$ has a single $AF$-region.
Fix $\rho>0$.
Pick $a\in T^\punc$ with $|a|>\rho$.
Since \( x_T \eaf x_{T-a}\), there exists $n_0$ such that  $\kappa_n^c(T)$ and $\kappa_n^c(T-a)$ belong to the same tile in $K^c(T_n)$ for all $n\ge n_0$ (see Proposition \ref{bratteli09.prop-projhulln}).
Therefore \(\dist( t_1, \partial t_n) >a > \rho\) for all $n\ge n_0$.
Since $\rho$ was arbitrary this proves that \(\lim_{n\rightarrow + \infty} \rho_n= +\infty\), {\it i.e.} that \(x_T\in G\).

Assume that $x\in G$.
Choose $a\in T_x^\punc$.
Since \(\lim_{n\rightarrow + \infty} \dist( t_1, \partial t_n) = + \infty\), there exists $n_0$ such that for all $n\ge n_0$ the patch $p_n=p_{x\vert_{\Pi_n}} \subset T_x$ contains a ball of radius $2aR/r$ around the origin (where $r,R,$ are the parameters of the Delone set $T_x^\punc$).
Therefore $T_x-a$ agree with $p_n -a$ on a ball of radius $aR/r$.
Hence $\kappa_n^c(T_x)$ and $\kappa_n^c(T_x-a)$ belong to the same tile in $K^c(T_n)$ for all $n\ge n_0$.
Hence \(r(x\vert_{\Pi_n}) = r\bigl( x_{T_x-a}\vert_{\Pi_n} \bigr)\) for all $n\ge n_0$.
So we have \(x \eaf x_{T_x-a}\), {\it i.e.} \(a \, \dot{\sim} \,0\).
Since $a$ was arbitrary, this shows that $\xi$ has a single $AF$-region.
\end{proof}

\begin{rem}
\label{bratteli09.rem-AForbit}
{\em Let \(R'_{AF}=\varphi^\ast(R_{AF})\) be the equivalence relation on $\Xi$ that is the image of the equivalence relation $R_{AF}$ induced by the homeomorphism of Theorem \ref{bratteli09.thm-isotrans}.
Proposition~\ref{bratteli09.prop-AFregion} shows that the $R'_{AF}$-orbit of a tiling $T\in\Xi$ is the set of all translates of $T$ by vectors linking to punctures that are in the $AF$-region of the origin:
\[
[T]_{AF} = \bigl\{ T -a \ : \ a \in T^\punc\,, \;    a \,\dot{\sim}\, 0
\bigr\}\,,
\]
So if $T$ has more than one $AF$-region, {\it i.e.} if \(x_T \in F\), then its $R'_{AF}$-orbit is only a proper subset of its $R_\Xi$-orbit (Definition \ref{bratteli09.eq-ERtrans}).
So we have:
\[
[T]_{AF} = [T]_{R_\Xi}  \iff x_T \in G\,.
\]
And for all $T$ in the dense subset \(\varphi(F)\subset \Xi\) we have \([T]_{AF} \subsetneq [T]_{R_\Xi}\).
}
\end{rem}

\subsection{Reconstruction of tiling groupoids}
\label{bratteli09.ssect-reconstruction}

As noted in Remark \ref{bratteli09.rem-AForbit}, the images in $\Xi$ of the $R_{AF}$-orbits do not always match those of $R_\Xi$.
In this section, we build a new equivalence relation on $\partial \Bb$ that ``enlarges'' $R_{AF}$, and from which we recover the full equivalence relation $R_\Xi$ on $\Xi$.
We consider a collared Bratteli diagram \(\Bb^c=(\Bb, \Hh)\) associated with a nested sequence \((T_n)_{n\in\NM}\) (Definition~\ref{bratteli09.def-tilbrat} and~\ref{bratteli09.def-cBrat}), and we denote by $\Xi$ the canonical transversal of the tiling space of $T_1$.

\begin{defini}
\label{bratteli09.def-comdiag}
A {\em commutative diagram} in $\Bb^c$ is a closed subgraph 
\[
\xymatrix{
\ar[dd]^{e} \ar[rr]^{h} & & \ar[dd]^{e'} \\
&& \\
\ar[rr]^{h'} &&
}
\] 
where $e,e'\in\Ee_n$, $h\in\Hh_{n-1}$, and $h'\in\Hh_n$, for some $n$, and such that
\[
\left\{
\begin{array}{lcl}
s(h)&= &s(e)\,,\\
r(h)&= &s(e')\,, \\
r(e)&= &s(h')\,, \\
r(e')&=&r(h')\,,
\end{array}
\right. \quad
\text{ and } \quad u(e)+u(h') = u(h)+u(e')\,.
\]
\end{defini}

Figure \ref{bratteli09.fig-comdiag} illustrates geometrically the conditions of adjacency required for tiles to fit into a commutative diagram.
\begin{figure}[!ht]
  \begin{center}
  \includegraphics[width=10cm]{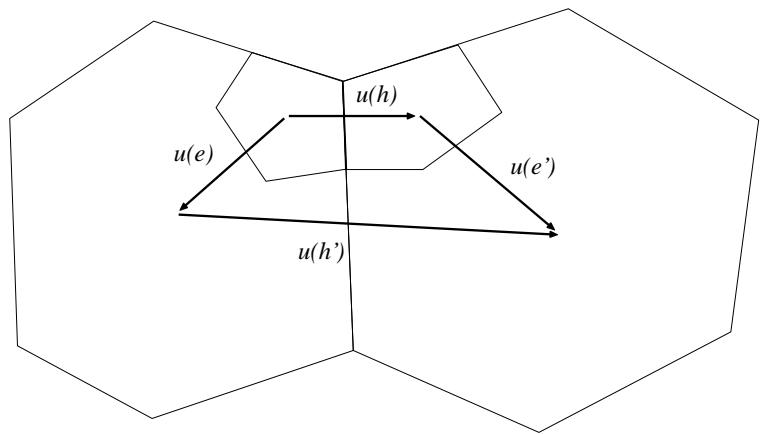}
  \end{center}
\caption{\small{Illustration of the commutative diagram in Definition \ref{bratteli09.def-comdiag}.}}
\label{bratteli09.fig-comdiag}
\end{figure}

With the notion of commutative diagram we can now define an equivalence relation on $\partial \Bb$ that contains $R_{AF}$.
\begin{defini}
\label{bratteli09.def-equiv}
We say that two infinite paths \(x=(e_n)_{n\in\NM}\) and \(y=(e'_n)_{n\in\NM}\) in $\partial \Bb$ are equivalent, and write \(x \, \sim \, y\), if there exists $n_0 \in \NM$, and $h_n \in \Hh_n$ for all $n\ge n_0$, such that for each $n>n_0$ the subgraph
\[
\xymatrix{
\ar[dd]^{e_n} \ar[rr]^{h_{n-1}} & & \ar[dd]^{e'_n} \\
&& \\
\ar[rr]^{h_n} &&
}
\] 
is a commutative diagram.
\end{defini}

\begin{lemma}
\label{bratteli09.lem-taileq}
If \( x\, \sim \, y\) in $\partial \Bb$, then there exists \(a(x,y)\in\RM^d\) such that \(T_y = T_x + a(x,y)\).
\end{lemma}
\begin{proof}
Let $n_0\in\NM$ be as in Definition \ref{bratteli09.def-equiv}, and for $n>n_0$ set \( a_n = u(x\vert_{\Pi_n}) - u(y\vert_{\Pi_n}) + u(h_n)\).
For all $n>n_0+1$ we have
\[
a_n = 
a_{n-1} -u(h_{n-1}) +  u(e_n) - u(e'_n) + u(h_n) = a_{n-1}\,,
\]
where the last equality occurs by commutativity of the diagram between generations $n-1$ and $n$.
Hence we have \(a_n = a_{n_0}\) for all $n>n_0$.
Now for all $n>n_0$, the patches \(p_{x\vert_{\Pi_n}}\) and \(p_{y\vert_{\Pi_n}} + a_{n_0}\) belong to $T_x$ (and similarly \(p_{y\vert_{\Pi_n}}, p_{x\vert_{\Pi_n}} - a_{n_0}\subset T_y\)).
Hence \(T_x = T_y + a_{n_0}\).
Set \(a(x,y)=-a_{n_0}\)  to complete the proof.
\end{proof}

\begin{defini}
\label{bratteli09.def-eqrel}
We define the equivalence relation on $\partial \Bb$
\[
R_\Bb = \bigl\{ (x,y) \in \partial \Bb \ : \ x \, \sim \, y
\bigr\}\,,
\]
with the following topology: \((x_n,y_n)_{n\in\NM}\) converges to $(x,y)$ in $R_\Bb$, if 
\((x_n)_{n\in\NM}\) converges to $x$ in $\partial\Bb$, and \(a(x_n,y_n) \rightarrow a(x,y)\) in $\RM^d$.
\end{defini}
The FLC and repetitivity properties of $T_1$ imply that for all $T\in\Xi$, the set of vectors linking its punctures, \(T^\punc-T^\punc\), equals \(T_1^\punc - T_1^\punc\), and is {\em discrete} and {\em closed} (see \cite{BHZ00,La99A,La99B,LP03} for instance).
The convergence \((x_n,y_n)\rightarrow (x,y)\) in $R_\Bb$ implies then that there exists $n_0\in\NM$ such that \(a(x_n,y_n)=a(x,y)\) for $n\ge n_0$.
\begin{rem}
\label{bratteli09.rem-AFcomp}
{\em If \(x\, \sim \, y\) in $R_\Bb$ are such that the horizontal edges $h_n\in\Hh_n$ of the commutative diagrams are all trivial, then for all $n\ge n_0$ we  have
\(r(x\vert_{\Pi_n}) = s(h_n) = r(h_n)=r(y\vert_{\Pi_n})\), {\it i.e.} $x$ and $y$ are tail equivalent in $\Bb$: \(x \eaf y\).
Thus we have the inclusion
\[
R_{AF} \subsetneq R_\Bb\,.
\]
In view of Remark \ref{bratteli09.rem-AForbit}, the two equivalence relation coincide on $G$, but differ on $F$ (see Proposition \ref{bratteli09.prop-Gdelta} and Corollary \ref{bratteli09.cor-Fsigma}) hence the inclusion is not an equality.
}
\end{rem}

We now give a technical lemma to exhibit a convenient base for the topology of $R_\Bb$.

Given two paths \(\gamma,\gamma'\in\Pi_n\), such that there exists $h\in\Hh_n$ with \(s(h)=r(\gamma)\) and \(r(h)=r(h)\) we define
\begin{equation}
\label{bratteli09.eq-a}
a_{\gamma \gamma'} = u(\gamma)-u(\gamma')+u(h)\,, \quad \text{and} \quad 
[\gamma\gamma'] = \varphi^{-1} \Bigl( 
\varphi([\gamma]) \cap \bigl( \varphi([\gamma']) - a_{\gamma\gamma'} \bigr)
\Bigr)\,,
\end{equation}
where $\varphi$ is the homeomorphism of Theorem \ref{bratteli09.thm-isotrans}.
So \([\gamma\gamma']\) is the clopen set of tilings in $\Xi$ which have the patch $p_\gamma$ at the origin {\em and} a copy of the patch $p_{\gamma'}$ at position $a_{\gamma\gamma'}$.

Recall from Definition \ref{bratteli09.def-etale} that an $R_\Bb$-set is an open set in $R_\Bb$ on which the source and range maps are homeomorphisms.

\begin{lemma}
\label{bratteli09.lem-baseRsets}
For $\gamma,\gamma'\in\Pi_n$, $n\in\NM$, the sets
\[
O_{\gamma\gamma'} = \bigl\{
(x,y) \in R_\Bb \ : \  x\in[\gamma\gamma']\,, \ a(x,y)=a_{\gamma\gamma'}
\bigr\}\,,
\]
form a base of $R_\Bb$-sets for the topology of $R_\Bb$.
\end{lemma}
\begin{proof}
We first prove that the $O_{\gamma\gamma'}$ form a base for the topology of $R_\Bb$.
A base open set in $R_\Bb$ reads 
\(O_{UV} = \{(x,y)\in R_\Bb\,:\, x\in U,\, a(x,y)\in V\}\)
for a clopen $U\subset\Xi$ and an open set $V\subset\RM^d$.
As noted after Definition \ref{bratteli09.def-eqrel}, the set of vectors $a(x,y)$ is a subset of the countable, discrete, and closed set \(T_1^\punc-T_1^\punc\).
Hence we can write \(O_{UV}\) as a countable union of open sets of the form 
\(O_{U,a}= \{(x,y)\in R_\Bb\,:\, x\in U,\, a(x,y)= a\}\)
for some \(a \in T_1^\punc-T_1^\punc\).
Since the sets \([\gamma],\gamma \in\Pi\), form a base for the topology of $\Xi$ we can write $O_{U,a}$ as a (finite) union of open sets of the form
\(O_{\gamma,a}= \{(x,y)\in R_\Bb\,:\, x\in [\gamma],\, a(x,y)= a\}\)
for some \(\gamma \in\Pi\).
And we can choose those $\gamma\in\Pi_n$ for $n$ large enough such that $a$ belongs to a puncture of a tile in $p_\gamma$, that is \(a=a(\gamma,\gamma')\) for some $\gamma'\in\Pi_n$, and thus \(O_{\gamma,a}=O_{\gamma\gamma'}\).
Note that this proves that the sets $[\gamma\gamma']$ also form a base for the topology of $\partial\Bb$.
Hence any open set in $R_\Bb$ is a union of $O_{\gamma\gamma'}$, and therefore the sets $O_{\gamma\gamma'}$ form a base for the topology of $R_\Bb$.

We now prove that $O_{\gamma\gamma'}$ is an $R_\Bb$-set.
By definition  $O_{\gamma\gamma'}$ is open, so it suffices to show that the maps  \(s\vert_{O_{\gamma\gamma'}}\) and \(r\vert_{O_{\gamma\gamma'}}\) are homeomorphisms.
First note that
\begin{equation}
\label{bratteli09.eq-bases}
s(O_{\gamma\gamma'})=[\gamma\gamma']\,, \quad \text{ and} \quad  r(O_{\gamma\gamma'})=[\gamma'\gamma] \,.
\end{equation}
Given \((y,z) \in O_{\gamma\gamma'}\), by Lemma \ref{bratteli09.lem-taileq}, we have \(T_z = T_y+a_{\gamma\gamma'}\).
Hence given $y \in [\gamma\gamma']$, there is a unique $z \in [\gamma'\gamma]$ such that \((y,z)\in O_{\gamma\gamma'}\).
And similarly, given $z \in [\gamma'\gamma]$, there is a unique $y \in [\gamma\gamma']$ such that \((y,z)\in O_{\gamma\gamma'}\).
Therefore the maps \(s\vert_{O_{\gamma\gamma'}}\) and \(r\vert_{O_{\gamma\gamma'}}\) are one-to-one.
But equation \eqref{bratteli09.eq-bases} shows that they map base open sets in $R_\Bb$ to base open set in $\partial \Bb$.
Hence those maps are homeomorphisms.
\end{proof}

We now state the main theorems, which characterize the equivalence relation $R_\Bb$, and compare it with $R_\Xi$.
\begin{theo}
\label{bratteli09.thm-etale}
The equivalence relation $R_\Bb$ is {\em \'etale}.
\end{theo}
\begin{proof}
We check conditions (i), (ii), and (iii) of Definition \ref{bratteli09.thm-etale}.

(iii) Let us show first that the maps $r$ and $s$ are continuous.
It suffices to show that \(s^{-1}[\gamma]=\{(x,y)\,:\,x\in[\gamma],\, x\sim y\}\) and \(r^{-1}([\gamma])=\{(x,y)\,:\,y\in[\gamma],\, x\sim y\}\) are open in $R_\Bb$.
Pick $(x,y) \in s^{-1}[\gamma]$ (respectively $(x,y)\in r^{-1}[\gamma]$).
Since $x\sim y$, there exists $\gamma'$ such that \(x\in[\gamma\gamma']\subset[\gamma]\) (respectively \(y\in[\gamma'\gamma]\subset[\gamma]\)).
Since \(a(x,y)=a(\gamma,\gamma')\) we have \((x,y)\in O_{\gamma\gamma'} \subset s^{-1}([\gamma])\) (respectively \((x,y)\in O_{\gamma\gamma'} \subset r^{-1}([\gamma])\)).
Thus \(s^{-1}([\gamma])\) (respectively \(r^{-1}([\gamma])\)) is open.

We have showed in Lemma \ref{bratteli09.lem-baseRsets} that sets $O_{\gamma\gamma'}$ are base $R_\Bb$-sets.
Hence the maps $r$ and $s$ are local homeomorphisms. 
From equation \eqref{bratteli09.eq-bases}, we see that they are also open.

(i) Pick \(w=((x_1,x_2),(x_3,x_4))\), with $x_3\ne x_3$, in \(R_\Bb \times R_\Bb \setminus R_\Bb^2\).
Let \(\gamma'_2, \gamma'_3 \in\Pi_n\) be such that for all \((x,y)\in[\gamma'_2]\times [\gamma'_3]\) we have $x\ne y$.
Choose $\gamma_1,\gamma_2 \in\Pi_m$, $m\ge n$, with $\gamma_2\vert_{\Pi_n}=\gamma'_2$, such that \((x_1,x_2)\in O_{\gamma_1\gamma_2}\).
And choose similarly  $\gamma_3,\gamma_4 \in\Pi_l$, $l\ge n$, with $\gamma_3\vert_{\Pi_n}=\gamma'_3$, such that \((x_3,x_4)\in O_{\gamma_3\gamma_4}\).
The set \(O_{\gamma_1\gamma_2} \times  O_{\gamma_3\gamma_4}\) is open in \(R_\Bb \times R_\Bb \setminus R_\Bb^2\) and contains $w$.
Hence \(R_\Bb \times R_\Bb \setminus R_\Bb^2\) is open, and therefore \(R_\Bb^2\) is closed in \(R_\Bb \times R_\Bb\).

Call $p_1$ the map that sends \(((x,y),(y,z))\) to $(y,z)$, $p_2$ that which sends it to $(x,z)$, and $p_3$ that which sends it to $(x,y)$.
We have
\[
\left\{
\begin{array}{rcl}
p_1^{-1}(O_{\gamma\gamma'}) & = & s^{-1}([\gamma'\gamma]) \times O_{\gamma\gamma'}\\
p_2^{-1}(O_{\gamma\gamma'}) & = & r^{-1}([\gamma\gamma']) \times s^{-1}([\gamma'\gamma])\\
p_3^{-1}(O_{\gamma\gamma'}) & = & O_{\gamma\gamma'} \times r^{-1}([\gamma'\gamma]) 
\end{array}
\right. \,,
\]
and the sets on the right hand sides are all open sets in $R_\Bb^2$.
Hence the maps $p_1,p_2$, and $p_3$ are continuous.

(ii) Let \((x,y)\in R_\Bb \setminus \Delta(R_\Bb)\), so we have $x\ne y$.
For each $n$ pick \(x_n \in [x\vert_{\Pi_n}]\), and define \(y_n\in\partial\Bb\) to coincide with $y$ on $\Pi_n$, and with $x_n$ on its tail.
Since $x\ne y$ we have $x_n\ne y_n$, hence \((x_n,y_n)\in  R_\Bb \setminus \Delta(R_\Bb)\), for all $n\ge n_0$ for some $n_0$.
Since $(x,y)\in R_\Bb$, by Lemma \ref{bratteli09.lem-taileq}, there exists $n_1$ such that for $n>n_1$ we have \(a(x\vert_{\Pi_n},y\vert_{\Pi_n}) = a(x,y)\).
Set \(n_2 = \max(n_0,n_1)+1\).
We have proved that the sequence \((x_n,y_n)_{n\ge n_2}\) has all its elements in \(R_\Bb \setminus \Delta(R_\Bb)\), and is such that: \(x_n \rightarrow x\) in $\partial \Bb$, and \(a(x_n,y_n)=a(x,y)\).
Therefore it converges to $(x,y)$ in \(R_\Bb \setminus \Delta(R_\Bb)\).
This proves that \(R_\Bb \setminus \Delta(R_\Bb)\) is closed in $R_\Bb$, hence that \(\Delta(R_\Bb)\) is open in $R_\Bb$.
\end{proof}

\begin{theo}
\label{bratteli09.thm-eqrel}
The two equivalence relations $R_\Bb$ on $\partial \Bb$, and $R_\Xi$ on $\Xi$, are homeomorphic:
\[
R_\Bb \ \cong \ R_\Xi \,.
\]
The homeomorphism is induced by \(\varphi : \partial \Bb \rightarrow \Xi\) from Theorem \ref{bratteli09.thm-isotrans}.
\end{theo}
\begin{proof}
Consider the map \( \varphi^\ast : R_\Bb \rightarrow R_\Xi\), given by
\[
\varphi^\ast (x,y) = \bigl( T_x, T_y = T_x+a(x,y) \bigr) \,.
\]
Since $\varphi$ is a homeomorphism, \(\varphi^\ast\) is injective.
To prove that is surjective, consider \((T,T'=T+a)\in R_\Xi\) and let us show that \((x_T,x_{T'})\) belongs to $R_\Bb$, {\it i.e.} that $x_T\sim x_{T'}$.
For each $n\in\NM$, call $t_n, t'_n$, the tiles in $\Vv_n$ such that \(\kappa_n^c(T) \in t_n\) and \(\kappa_n^c(T') \in t'_n\) (see Proposition \ref{bratteli09.prop-projhulln}).
The nested sequence of tilings $(T_n)_{n\in\NM}$ induces a nested sequence $(\tilde{T}_n)_{n\in\NM}$ with \(\tilde{T}_1 = T\).
In $\tilde{T}_n$ the origin lies in a translate $\tilde{t}_n$ of $t_n$, and the point $a$ in a translate $\tilde{t}'_n$ of $t'_n$.
Since $r_n\rightarrow \infty$ (condition (iii) in Definition \ref{bratteli09.def-nested}), there exists $n_0$ such that for all $n>n_0$ those two tiles  $\tilde{t}_n$ and  $\tilde{t}'_n$ are within a distance $r_n$ to one another: \(\dist(\tilde{t}_n, \tilde{t}'_n) \le r_n\).
Therefore we have \(\tilde{t}_n \in \col(\tilde{t}'_n)\) and \(\tilde{t}'_n \in  \col(\tilde{t}_n)\) (Definition \ref{bratteli09.def-ctile}).
This means that for all $n>n_0$ there exists a horizontal edge \(h_n\in\Hh_n\) with source \(t_n\) and range \(t'_n\) (Definition \ref{bratteli09.def-cBrat}).
As $x_T$ and $x_{T'}$ are the infinite paths in $\partial\Bb$  through the vertices \(t_n\) and \(t'_n\) respectively, we have $x_T \sim x_{T'}$.
This proves that $\varphi^\ast$ is a bijection.

Now a sequence \((x_n,y_n)_{n\in\NM}\) converges to \((x,y)\) in $R_\Bb$ if and only if \(x_n\rightarrow x\) in $\partial \Bb$, and \(a(x_n,y_n)\rightarrow a(x,y)\) in $\RM^d$.
This is the case if and only if \(T_{x_n} \rightarrow T_x = T_y+a(x,y)\) in $\Xi$ and \(a(x_n,y_n)\rightarrow a(x,y)\) in $\RM^d$, since for all $n$ one has \(T_{y_n} = T_{x_n} + a(x_n,y_n)\) by Lemma \ref{bratteli09.lem-taileq}.
Hence the map $\varphi^\ast$ and its inverse are continuous.
\end{proof}

\begin{coro} 
\label{bratteli09.cor-groupoid}
The groupoid of the equivalence relation $R_\Bb$ is homeomorphic to $\Gamma_\Xi$.
\end{coro}

\section{Examples: one-dimensional tilings}
\label{bratteli09.sect-examples}

We illustrate here our construction for dimension $1$ substitutions.
In this case, if one chooses the proper sequence of tilings according to the substitution, then one recovers the usual formalism of Bratteli diagrams associated with the (Abelianization matrix of the) substitution.
We show here that the horizontal structure introduced in Definition~\ref{bratteli09.def-cBrat} allows to recover natural minimal and maximal infinite paths, and that the equivalence relation $R_\Bb$ is then exactly generated by $R_{AF}$ and the set of minimal and maximal paths \((x_{min},x_{max}) \notin R_{AF}\) such that \(V(x_{max}) =x_{min}\), where $V$ is the Vershik map on $\Bb$ (corresponding to an associated ordering of the edges).
This will be shown carefully in examples, but this fact is general: the following result holds.
\begin{proposi}
\label{bratteli09.prop-1d}
Let $\Bb$ be a collared Bratteli diagram associated with a $1$-dimensional tiling, with labelled edges.
Then there is a partial order on edges, which induces a partial order on infinite paths.
Furthermore, there is a one-to-one map $\psi$ from the set of maximal paths to the set of minimal paths, such that:
\[
 R_\Bb = R_{AF} \  \wedge \bigcup_{x \text{ maximal path}}{(x,\psi(x))}.
\]
\end{proposi}
The fact that our labels on vertical edges give a partial ordering on edges is immediate:
given $v \in \Vv$, the set $r^{-1} (v)$ is a set of edges encoding the inclusions of tiles in the substitution of~$t_v$.
In dimension $1$, it makes sense to define the edge of $r^{-1}(v)$ which corresponds to the \emph{leftmost} tile included in the substitution of $t_v$ (it is the edge with the \emph{bigger} label in $\RM$).
A minimal (respectively maximal) path is then a path made uniquely of minimal (respectively maximal) edges.
It is then an exercise to show that given two paths in $R_\Bb \setminus R_{AF}$ then they are tail-equivalent to a maximal path, respectively a minimal path.
This gives a pairing of minimal with maximal paths, and thanks to the tiles decorations, this pairing is one-to-one.
The fact that the groupoid relation can be recovered from the $AF$ relation and a finite number of pairs is already known, and our formalism recovers this here.
Furthermore, the pairing $\psi$ corresponds actually to the translation of the associated tilings (more precisely to the action of the fist return map on $\Xi$).
This map, the Vershik map, can be read from the Bratteli diagram (from the partial ordering of vertices), see~\cite{DHS99} for example for the definition of this map.

We treat in details the cases of the Fibonacci and the Thue-Morse tilings.
Those tilings have been extensively studied, and we refer the reader to \cite{AP98} for a short presentation, and to \cite{GruShep87} for further material.

Both tiling spaces are strongly aperiodic, repetitive, and FLC \cite{AP98}.
This implies that the substitution induces a homeomorphism on the tiling space.
Let \((\Omega,\sigma)\) denote either the Fibonacci or Thue-Morse tiling space, $\Xi$ its canonical transversal, and $\lambda$ the inflation constant of the substitution.
Fix $T\in\Xi$.
We can build a nested sequence of tilings associated with $T$.
For $k$ large enough, the nested sequence \((\lambda^{kn} \sigma^{-kn}(T))_{n\in\NM}\) is proper, because the $k$-th substitute of any tile contains one in its interior.
One can easily see however that condition (ii) in Definition~\ref{bratteli09.def-nested} can be replaced by the weaker assumption:

{\em For all $n$ large enough and each tile $t_n\in T_n$ there exists $m<n$ such that \(\sigma_{nm}(t_n)\) contains a tile of $T_m$ in its interior}.

For example, it is straightforward to see that the proof of Theorem~\ref{bratteli09.thm-isotrans} goes through with this weaker condition.
As noted earlier, we can take here $m=n-k$ independently of the tile, for some fixed $k$.

We will thus consider the sequence \((T_n)_{n\in\NM}\), with \(T_n=\lambda^n \sigma^{-n}(T)\).
The Bratteli diagram built in Definition~\ref{bratteli09.def-tilbrat}, without the label $u$, is therefore exactly the usual Bratteli diagram associated with a substitution (its Abelianization matrix).
For example, the encoding of words by paths, corresponds also exactly to the encoding of patches given in equation~\eqref{bratteli09.eq-pathpatch}.

\subsection{The Fibonacci tiling}
\label{bratteli09.ssect-Fibo}

Let \(\Omega\) be the Fibonacci tiling space, and $\Xi$ its canonical transversal.
Each tiling in $\Omega$ has two types of tiles up to translation, denoted $0$ and $1$.
The prototile $0$ is identified with the closed interval $[-1/2,1/2]$ with puncture at the origin, and the prototile $1$ is identified with $[-1/(2\phi),1/(2\phi)]$ with puncture at the origin, where \(\phi = (1+ \sqrt{5})/2\) is the golden mean.

The substitution is given by \(0\rightarrow 01\), \(1 \rightarrow 0\), and its inflation constant  is $\phi$.
We write $a,b,c,d,$ for the collared prototiles, where 
\[
\col(a) = 0 \dot{0} 1 \,, \quad 
\col(b) = 1 \dot{0} 0 \,, \quad
\col(c) = 1 \dot{0} 1 \,, \quad
\col(d) = 0 \dot{1} 0 \,,
\]
and where the dot indicates the tile that holds the puncture.
So $a,b$, and $c$ correspond to the tile $0$ but with different labels, while $d$ corresponds to $1$.
The substitution on collared tiles reads then
\[
\sigma(a)= cd\,, \quad
\sigma(b)= ad\,, \quad
\sigma(c)= ad\,, \quad
\sigma(d)= b\,.
\]

Let $T\in\Xi$, and consider the sequence \((T_n)_{n\in\NM}\), with \(T_n=\phi^n \sigma^{-n}(T)\).
The Bratteli diagram $\Bb$ associated with this proper sequence has then the following form between two generations (excluding the root):
\begin{equation}
\label{bratteli09.eq-fibo}
\xymatrix{
a_{n-1} \ar@{-}[ddrr]^(.9){2} \ar@{-}[ddrrrr]^(.9){2} && b_{n-1} \ar@{-}[ddrrrr]^(.9){1} && 
c_{n-1} \ar@{-}[ddllll]_(.9){2} && d_{n-1} \ar@{-}[ddllllll]^(.9){1} \ar@{-}[ddllll]_(.9){1} \ar@{-}[ddll]_(.9){1}\\
&& && && \\
a_n && b_n && c_n && d_n \\
}
\end{equation}
where $t_n$ is the support of \(\col (\sigma^n (t))\) for \(t=a,b,c,d\).
We write an arrow as \(e^n_{tt'}\in\Ee_n\) with \(s(e^n_{tt'})=t_{n-1}\) and \(r(e^n_{tt'})=t'_n\).
We have for $n\ge 2$
\[
u(e^n_{ab}) =u(e^n_{ac}) = u(e^n_{ca}) = \frac{1}{2\phi} \phi^{n-2}\,, \quad
u(e^n_{bd}) = 0\,, \quad
u(e^n_{da}) = u(e^n_{db}) = u(e^n_{dc}) = - \frac{1}{2} \phi^{n-2}\,.  
\]
Note that given a path \(\gamma\in\Pi\), the patch $p_\gamma$ of equation \eqref{bratteli09.eq-pathpatch} corresponds exactly to the word associated with that path in the usual formalism of the Bratteli diagram of a substitution.
For example, for the path \((e^1_a, e^2_{ac},e^3_{ca},e^4_{ab})\), we associate the patch or word \(\sigma^3(b)\rightarrow \sigma^2(\dot{a}d) \rightarrow  \sigma(\dot{c}db) \rightarrow \dot{a}dbad\).

The horizontal graph $\Hh_n$ has the form
\vspace{.5cm}
\[
\xymatrix{
a_{n} \ar@(ur,ul)@{-}[rr]^{ba} \ar@(d,d)@{-}[rrrrrr]^{ad} && b_{n} \ar@(u,u)@{-}[rrrr]^{db} && 
c_{n} \ar@(ur,ul)@{-}[rr]_{cd} \ar@(dr,dl)@{-}[rr]^{dc} && d_{n} 
}
\]
\vspace{.5cm}

where the indices show the patches corresponding to the edges (with left-right orientation).
We have not shown trivial edges, and have identified an edge and its opposite in the drawing.
If $tt'$ is one of the patches (with this orientation) in the above graph, we write an horizontal edge \(h^n_{tt'} \in\Hh_n\) with \(s(h^n_{tt'})=t'_{n-1}\) and \(r(h^n_{tt'})=t_n\).
And we will simply use \((h^n_{tt'})^\op\) to avoid confusions (for example between \((h^n_{cd})^\op\) and \(h^n_{dc}\)).
For $n\ge 1$ we have
\[
u(h^n_{ab}) = - \phi^{n-1}\,, \quad
u(h^n_{ad}) =  u(h^n_{cd}) = u(h^n_{dc}) = u(h^n_{db}) = - \frac{\phi}{2} \phi^{n-1} \,.
\]

There are only two commutative diagrams that one can write between two generations, namely
\begin{equation}
\label{bratteli09.eq-cdiagfibo}
\xymatrix{
b_{n-1} \ar[dd]^{e^n_{bd}} && \ar[ll]_{h^{n-1}_{ba}} \ar[dd]_{e^n_{ac}} a_{n-1} &&&
d_{n-1} \ar[dd]^{e^n_{db}} && \ar[ll]_{h^{n-1}_{dc}} \ar[dd]_{e^n_{ca}} c_{n-1} \\
&& &&& \\
d_n && \ar[ll]^{h^n_{cd}} c_n &&&
b_n && \ar[ll]^{h^n_{ab}} a_n 
}
\end{equation}
The translation do match in those diagrams, we have: \(u(e^n_{bd}) + u(h^{n-1}_{ba}) = u(e^n_{ac})+u(h^n_{cd})=  - \phi^{n-2}\), and \(u(e^n_{db}) + u(h^{n-1}_{dc}) = u(e^n_{ca})+u(h^n_{ab})=  - \phi^{n-2}(\phi+1)/2\).

Let us write $D_n^1$ and $D_n^2$ for the left and right above diagrams respectively.
The top horizontal edge of $D_{n+1}^1$ matches the bottom horizontal edge of $D_{n}^2$, and the top horizontal edge of $D_{n+1}^2$ matches the bottom one of $D_n^1$.
We can thus ``compose'' those diagrams, and consider the two infinite sequences \((D_2^1, D_3^2, D_4^1, \cdots D_{2n}^1, D_{2n+1}^2, \cdots)\) and \((D_2^2, D_3^1, D_4^2, \cdots D_{2n}^2, D_{2n+1}^1, \cdots)\).
Each of those sequences contains exactly two infinite paths, namely:
\begin{align*}
x^1_{min}& =(e^1_{a},e^2_{ac},e^3_{ca}, \cdots e^{2n}_{ac}, e^{2n+1}_{ca}, \cdots) &
x^1_{max}& =(e^1_{b},e^2_{bd},e^3_{db}, \cdots e^{2n}_{bd}, e^{2n+1}_{db}, \cdots) \\
x^2_{min}& =(e^1_{c},e^2_{ca},e^3_{ac}, \cdots e^{2n}_{ca}, e^{2n+1}_{ac}, \cdots) &
x^2_{max}& =(e^1_{d},e^2_{db},e^3_{bd}, \cdots e^{2n}_{db}, e^{2n+1}_{bd}, \cdots)
\end{align*}
(where we have added an edge to the root).
If we order the edges in $\Bb$ as shown in equation~\eqref{bratteli09.eq-fibo}, those are the two minimal and maximal infinite paths.
And if we let $V$ denote the Vershik map on $\partial \Bb$, we have \(V(x^i_{max})=x^i_{min}\), for $i=1,2$.

All those paths are pairwise non-equivalent in $R_{AF}$, but by definition they are equivalent in $R_\Bb$.
Now given two infinite paths \(x,y \in \partial \Bb\) such that \(x\eaf x^i_{min}\) and \(y\eaf x^i_{max}\), for $i=1$ or $2$, we have \(x\sim y\) in $R_\Bb$.
We have thus shown that $R_\Bb$ is generated by $R_{AF}$ and the two pairs \((x^1_{min},x^1_{max})\) and \((x^2_{min},x^2_{max})\):
\[
R_\Bb = R_{AF} \wedge \{(x^1_{min},x^1_{max}),(x^2_{min},x^2_{max}) \} \,.
\]

\subsection{The Thue-Morse tiling}
\label{bratteli09.ssect-TM}

Let \(\Omega\) be the Thue-Morse tiling space, and $\Xi$ its canonical transversal.
Each tiling has two types of tiles up to translation, denoted $0$ and $1$.
Each prototile is identified with the closed interval $[-1/2,1/2]$ with puncture at the origin.

The substitution is given by \(0\rightarrow 01\), \(1 \rightarrow 10\), and its inflation constant  is $2$.
We write $a,b,c,d,e,f,$ for the collared prototiles, where 
\begin{align*}
\col(a) &= 0 \dot{0} 1 &
\col(b) &= 1 \dot{0} 0 & 
\col(c) &= 0 \dot{1} 1 \\
\col(d) &= 1 \dot{1} 0 &
\col(e) &= 1 \dot{0} 1 &
\col(f) &= 0 \dot{1} 0
\end{align*}
and where the dot indicates the tile that holds the puncture.
So $a,b$, and $e$ correspond to the tile $0$ but with different labels, while $c,d,$ and $f$ correspond to $1$.
The substitution on collared tiles reads then
\[
\sigma(a)= bf\,, \quad
\sigma(b)= ec\,, \quad
\sigma(c)= de\,, \quad
\sigma(d)= fa\,, \quad
\sigma(e)= bc\,, \quad
\sigma(f)= da\,.
\]

Let $T\in\Xi$, and consider the nested sequence \((T_n)_{n\in\NM}\), with \(T_n= 2^n \sigma^{-n}(T)\).
The Bratteli diagram $\Bb$ associated with that sequence has then the following form between two generations (excluding the root):
\begin{equation}
\label{bratteli09.eq-TM}
\xymatrix{
a_{n-1} \ar@{-}[ddrrrrrr]^(.9){2} \ar@{-}[ddrrrrrrrrrr]_(.92){2} && 
b_{n-1} \ar@{-}[ddll]_(.9){1} \ar@{-}[ddrrrrrr]_(.9){1}  && 
c_{n-1} \ar@{-}[ddll]_(.94){2} \ar@{-}[ddrrrr]^(.9){2} && 
d_{n-1} \ar@{-}[ddll]_(.9){1} \ar@{-}[ddrrrr]^(.9){1} && 
e_{n-1} \ar@{-}[ddllllll]^(.9){1} \ar@{-}[ddllll]^(.9){2} && 
f_{n-1} \ar@{-}[ddllllllllll]^(.92){2} \ar@{-}[ddllll]_(.9){1} \\
&& && && && && \\
a_n && b_n && c_n && d_n && e_n && f_n
}
\end{equation}
where $t_n$ is the support of \(\col (\sigma^n (t))\) for \(t=a,b,c,d,e,f\).
We write an arrow as \(e^n_{tt'}\in\Ee_n\) with \(s(e^n_{tt'})=t_{n-1}\) and \(r(e^n_{tt'})=t'_n\).
We have for $n\ge 2$
\begin{align*}
u(e^n_{ba}) = u(e^n_{be}) = u(e^n_{dc}) = u(e^n_{df}) = u(e^n_{eb}) = u(e^n_{fd})            
& = \frac{1}{2}2^{n-2}\,, \\
u(e^n_{ad}) = u(e^n_{af}) = u(e^n_{cb}) = u(e^n_{ce}) = u(e^n_{ec}) =  u(e^n_{fa})
& = -\frac{1}{2} 2^{n-2}\,.
\end{align*}
Note that given a path \(\gamma\in\Pi\), the patch $p_\gamma$ of equation \eqref{bratteli09.eq-pathpatch} corresponds exactly to the word associated with that path in the usual formalism of the Bratteli diagram of a substitution.
For example, for the path \((e^1_a, e^2_{ad},e^3_{dc},e^4_{cb})\), we associate the patch or word \(\sigma^3(b)\rightarrow \sigma^2(e\dot{c}) \rightarrow  \sigma(bc\dot{d}e) \rightarrow ecdef\dot{a}bc\).

The horizontal graph $\Hh_n$ has the form
\vspace{.5cm}
\[
\xymatrix{
a_{n} \ar@{-}[rr]^{ab} \ar@{-}[ddrrrr]^{da} \ar@{-}[dd]_{fa}&& 
b_{n} \ar@{-}[rr]^{bc} \ar@{-}[ddll]^{bf} && 
c_{n} \ar@{-}[dd]^{cd} \ar@{-}[ddll]^{ec} \\
&& && \\ 
f_{n} \ar@(ur,ul)@{-}[rr]_{fe} \ar@(dr,dl)@{-}[rr]^{ef} &&
e_n   \ar@{-}[rr]_{de} && 
d_n
}
\]
\vspace{.5cm}

where the indices show the patches corresponding to the edges (with left-right orientation).
We have not shown trivial edges, and have identified an edge and its opposite in the drawing.
If $tt'$ is one of the patches (with this orientation) in the above graph, we write an horizontal edge \(h^n_{tt'} \in\Hh_n\) with \(s(h^n_{tt'})=t'_{n-1}\) and \(r(h^n_{tt'})=t_n\).
And we will simply use \((h^n_{tt'})^\op\) to avoid confusions (for example between \((h^n_{ef})^\op\) and \(h^n_{fe}\)).
For $n\ge 1$,  we have
\[
u(h^n_{tt'}) = -2^{n-1}\,.
\]

There are only four commutative diagrams that one can write between two generations, namely
\begin{equation}
\label{bratteli09.eq-cdiagTM1}
\xymatrix{
a_{n-1} \ar[dd]^{e^n_{af}} && \ar[ll]_{h^{n-1}_{ab}} \ar[dd]_{e^n_{be}} b_{n-1} &&&
f_{n-1} \ar[dd]^{e^n_{fa}} && \ar[ll]_{h^{n-1}_{fe}} \ar[dd]_{e^n_{eb}} e_{n-1} \\
&& &&& \\
f_n && \ar[ll]^{h^n_{fe}} e_n &&&
a_n && \ar[ll]^{h^n_{ab}} b_n 
}
\end{equation}
and 
\begin{equation}
\label{bratteli09.eq-cdiagTM2}
\xymatrix{
c_{n-1} \ar[dd]^{e^n_{ce}} && \ar[ll]_{h^{n-1}_{cd}} \ar[dd]_{e^n_{df}} d_{n-1} &&&
e_{n-1} \ar[dd]^{e^n_{ec}} && \ar[ll]_{h^{n-1}_{ef}} \ar[dd]_{e^n_{fd}} f_{n-1} \\
&& &&& \\
e_n && \ar[ll]^{h^n_{ef}} f_n &&&
c_n && \ar[ll]^{h^n_{cd}} d_n 
}
\end{equation}

The translation do match in those diagrams, for example we have: \(u(e^n_{af}) + u(h^{n-1}_{ab}) = u(e^n_{be})+u(h^n_{fe})= -(3/2) 2^{n-2}\), and \(u(e^n_{ce}) + u(h^{n-1}_{cd}) = u(e^n_{df})+u(h^n_{ef})= -(3/2) 2^{n-2}\).

Let us write $D_n^1$ and $D_n^2$ for the left and right diagrams in equation~\eqref{bratteli09.eq-cdiagTM1} respectively, and  $D_n^3$ and $D_n^4$ for those on the left and right of equation~\eqref{bratteli09.eq-cdiagTM2} respectively.
The diagrams $D_n^1$ and $D_n^2$, and $D_n^3$ and $D_n^4$ are ``composable'':  the top horizontal edge of one at generation $n+1$ matches the bottom horizontal edge of the other at generation $n$.
We consider the four infinite sequences \((D_2^1, D_3^2, D_4^1, \cdots D_{2n}^1, D_{2n+1}^2, \cdots)\) and \((D_2^2, D_3^1, D_4^2, \cdots D_{2n}^2, D_{2n+1}^1, \cdots)\), as well as  \((D_2^3, D_3^4, D_4^3, \cdots D_{2n}^3, D_{2n+1}^4, \cdots)\) and \((D_2^4, D_3^3, D_4^4, \cdots D_{2n}^4, D_{2n+1}^3, \cdots)\).
Each of those sequences contains exactly two infinite paths, namely:
\begin{align*}
x^1_{min}& =(e^1_{b},e^2_{be},e^3_{eb}, \cdots e^{2n}_{be}, e^{2n+1}_{eb}, \cdots) &
x^1_{max}& =(e^1_{a},e^2_{af},e^3_{fa}, \cdots e^{2n}_{af}, e^{2n+1}_{fa}, \cdots) \\
x^2_{min}& =(e^1_{e},e^2_{eb},e^3_{be}, \cdots e^{2n}_{eb}, e^{2n+1}_{be}, \cdots) &
x^2_{max}& =(e^1_{f},e^2_{fa},e^3_{af}, \cdots e^{2n}_{fa}, e^{2n+1}_{af}, \cdots) \\
x^3_{min}& =(e^1_{d},e^2_{df},e^3_{fd}, \cdots e^{2n}_{df}, e^{2n+1}_{fd}, \cdots) &
x^3_{max}& =(e^1_{c},e^2_{ce},e^3_{ec}, \cdots e^{2n}_{ce}, e^{2n+1}_{ec}, \cdots) \\
x^4_{min}& =(e^1_{f},e^2_{fd},e^3_{df}, \cdots e^{2n}_{fd}, e^{2n+1}_{df}, \cdots) &
x^4_{max}& =(e^1_{e},e^2_{ec},e^3_{ce}, \cdots e^{2n}_{ec}, e^{2n+1}_{ce}, \cdots)
\end{align*}
(where we have added edges to the root).
If we order the edges in $\Bb$ as shown in equation~\eqref{bratteli09.eq-TM}, those are the four minimal and four maximal infinite paths.
And if we let $V$ denote the Vershik map on $\partial \Bb$, we have \(V(x^i_{max})=x^i_{min}\), for $i=1,2,3,4$.

All those paths are pairwise non-equivalent in $R_{AF}$, but by definition they are equivalent in $R_\Bb$.
Now given two infinite paths \(x,y \in \partial \Bb\) such that \(x\eaf x^i_{min}\) and \(y\eaf x^i_{max}\), for $i=1,2,3$ or $2$, we have \(x\sim y\) in $R_\Bb$.
We have thus shown that $R_\Bb$ is generated by $R_{AF}$ and the two pairs \((x^i_{min},x^i_{max})\) for \(i=1,2,3,4\):
\[
R_\Bb = R_{AF} \wedge \{(x^1_{min},x^1_{max}),(x^2_{min},x^2_{max}),(x^3_{min},x^3_{max}),(x^4_{min},x^4_{max}) \} \,.
\]

\vspace{3 ex}

\hrule
\vspace{2 ex}

\begin{footnotesize}
\textsc{Jean Bellissard,} Georgia Institute of Technology, School of Mathematics and School of Physics\\
Postal address: Georgia Tech, School of Math., 686 Cherry street, Atlanta GA, 30332-0160, USA\\
\textsc{Antoine Julien}, Universit\'e de Lyon, Universit\'e Lyon 1, Institut Camille Jordan, UMR 5208 du CNRS, 43 boulevard du 11 novembre 1918, F-69622 Villeurbanne cedex, France \\
\textsc{Jean Savinien}, Universit\'e de Lyon, Universit\'e Lyon 1, Institut Camille Jordan, UMR 5208 du CNRS, 43 boulevard du 11 novembre 1918, F-69622 Villeurbanne cedex, France
\end{footnotesize}



\end{document}